\documentclass[11pt]{article}

\usepackage[margin=1.15in]{geometry}
\usepackage{amsmath,amssymb,amsthm,mathtools}
\usepackage{mathrsfs}
\usepackage{microtype}
\usepackage[hidelinks]{hyperref}
\usepackage{enumitem}

\hypersetup{
 pdftitle={Three-term asymptotics for discrete Hardy--Rellich constants in high dimension},
 pdfauthor={Carlos Lizama},
 pdfsubject={Spectral asymptotics of discrete Hardy--Rellich inequalities in high dimension},
 pdfkeywords={discrete Hardy--Rellich inequality, optimal constants,
 high-dimensional asymptotics,
 degenerate spectral cluster, Feshbach--Schur map, hyperoctahedral symmetry,
 weighted torus estimate}
}

\newtheorem{theorem}{Theorem}[section]
\newtheorem{proposition}[theorem]{Proposition}
\newtheorem{lemma}[theorem]{Lemma}
\newtheorem{corollary}[theorem]{Corollary}

\theoremstyle{remark}
\newtheorem{remark}[theorem]{Remark}

\newcommand{\Z}{\mathbb Z}
\newcommand{\T}{\mathbb T}
\newcommand{\N}{\mathbb N}
\newcommand{\cH}{\mathcal H}
\newcommand{\cE}{\mathcal E}
\newcommand{\cO}{\mathcal O}

\newcommand{\BN}{\mathfrak B_N}
\newcommand{\ip}[2]{\left\langle #1,#2\right\rangle}
\newcommand{\norm}[1]{\left\lVert #1\right\rVert}
\newcommand{\1}{\mathbf 1}

\title{Three-term asymptotics for discrete\\
Hardy--Rellich constants in high dimension}
\author{Carlos Lizama\\
\small Departamento de Matem\'atica y Ciencia de la Computaci\'on,\\[-1mm]
\small Facultad de Ciencia, Universidad de Santiago de Chile,\\[-1mm]
\small Las Sophoras 173, Estaci\'on Central, Santiago, Chile\\[-1mm]
\small \texttt{carlos.lizama@usach.cl}}
\date{}

\begin{document}
\maketitle

\begin{abstract}
Sharp Hardy--Rellich constants are spectral thresholds for operators with
critical inverse-power potentials.  Let $C_\ell(N)$ denote the optimal
constant in the $\ell$th-order discrete Hardy--Rellich inequality on $\Z^N$.
Recent independent work established the leading high-dimensional behavior
$C_\ell(N)\sim2^\ell N^\ell$ for every fixed $\ell$.  We determine the next
two orders and prove
\[
 C_\ell(N)=2^\ell N^\ell+\gamma_\ell N^{\ell-1}
 +\delta_\ell N^{\ell-2}+O_\ell(N^{\ell-3}),
\]
with explicit coefficients $\gamma_\ell$ and $\delta_\ell$.  The central
difficulty is a degeneracy that grows with the dimension: after conjugation,
the leading operator is scalar on the $2N$ nearest neighbours of the origin.
We resolve this cluster using signed-permutation symmetry and an effective
operator on four lattice orbits.  Weighted torus estimates and
Feshbach--Schur reduction justify the finite-dimensional expansion inside the
full operator, while a residual bound and Temple's inequality give the stated
remainder.  In particular,
$C_1(N)=2N-4-20/(3N)+O(N^{-2})$.
\end{abstract}



\section{Introduction and main result}

Hardy inequalities quantify the competition between an energy and a critical
singular potential.  Their Euclidean prototype is
\[
 \int_{\mathbb R^N}|\nabla f|^2dx
 \geq\frac{(N-2)^2}{4}
 \int_{\mathbb R^N}\frac{|f|^2}{|x|^2}dx,
 \qquad f\in C_c^\infty(\mathbb R^N),
\]
where $N\geq3$.  Equivalently, $(N-2)^2/4$ is the largest coupling for which
the quadratic form of
\(
 -\Delta-\lambda|x|^{-2}
\)
is non-negative.  This threshold interpretation connects sharp Hardy
constants with criticality theory, singular Schr\"odinger operators,
eigenvalue estimates, and heat-kernel bounds; see, among many sources,
\cite{BalinskyEvansLewis,BogdanEtAl2019,DevyverFraasPinchover,
FrankLiebSeiringer}.  The same singularity that makes the constant sharp also
prevents attainment in the natural energy space and makes its approximation
delicate; a recent quantitative result in this direction is
\cite{IgnatZuazua2025}.

Although the differential inequality is now the most familiar formulation,
Hardy's original argument arose from a discrete inequality for sequences; see
\cite{Hardy1920,KufnerMaligrandaPersson} for the historical development.  The
higher-order theory begins with Rellich and Birman.  For example, when $N\geq5$,
\[
 \int_{\mathbb R^N}|\Delta f|^2dx
 \geq \left(\frac{N(N-4)}4\right)^2
 \int_{\mathbb R^N}\frac{|f|^2}{|x|^4}dx.
\]
More generally, higher-order Hardy--Rellich inequalities have sharp constants
given by explicit products; see \cite{DaviesHinz,Yafaev}.  For each fixed
differential order $\ell$, these constants grow like $N^{2\ell}$ as the
dimension tends to infinity.

The present paper concerns the corresponding high-dimensional problem on the
integer lattice.  Let
\[
 (L_Nu)(n)=2Nu(n)-\sum_{|m-n|=1}u(m)
\]
be the non-negative combinatorial Laplacian on $\Z^N$.  For a fixed integer
$\ell\geq1$, we consider the optimal constant
\begin{equation}\label{eq:Cintro}
 C_\ell(N)=
 \inf_{\substack{u\in C_c(\Z^N),\ u(0)=0\\u\ne0}}
 \frac{\ip{u}{L_N^\ell u}}
 {\displaystyle\sum_{n\ne0}|u(n)|^2/|n|^{2\ell}}.
\end{equation}
Thus $C_\ell(N)$ is the critical coupling for the non-negativity, on the
punctured lattice, of
\[
 L_N^\ell-\lambda |n|^{-2\ell}.
\]
This places \eqref{eq:Cintro} within the spectral theory of discrete
Schr\"odinger operators; for general background on graph Laplacians and
discrete Dirichlet forms, and for spectral estimates on $\Z^N$, respectively,
see \cite{KellerLenzWojciechowski,RozenblumSolomyak2009}.

The lattice is not merely a finite-difference proxy for the continuum.  It has
no dilation symmetry, the origin has precisely $2N$ nearest neighbours, and
there is no radial integration-by-parts formula reducing \eqref{eq:Cintro} to
a one-dimensional variational problem.  These geometric features change the
dimensional scale of the sharp constant and permit near-minimizers to
concentrate on the first few coordinate shells.  The resulting contrast is
already visible at leading order: the lattice constant grows like $N^\ell$,
not $N^{2\ell}$.

Several complementary forms of discrete Hardy theory have developed around
this phenomenon.  For prescribed inverse-power weights, continuous and
discrete inequalities were compared in \cite{KapitanskiLaptev}, while recent
fractional and non-fractional lattice estimates appear in \cite{Dyda2025}.
Criticality and positive-supersolution methods instead construct an optimal
weight adapted to a given graph operator; see
\cite{DevyverFraasPinchover,KPP2018,KPP2021} and the monograph
\cite{KellerLenzWojciechowski}.  On the Euclidean lattice, Keller and Lemm
obtained precise large-distance behavior of optimal Hardy weights
\cite{KellerLemm2023}; nonlinear graph extensions were subsequently developed
in \cite{Fischer2024}.  The optimal-weight problem and the scalar problem
\eqref{eq:Cintro} are related but distinct.  The former identifies the
critical decay at infinity, whereas the latter fixes the Euclidean
inverse-power weight and is sensitive to the global lattice geometry.

The one-dimensional higher-order problem has also undergone substantial
development.  Sharp and improved discrete Hardy and Rellich inequalities were
proved in \cite{GerhatKrejcirikStampach2025,Gupta2024,HuangYe2022,
HuangYe2024,KrejcirikStampach2022}; optimal
Hardy--Rellich--Birman weights of arbitrary order were obtained in
\cite{StampachWaclawek2025}.  These results reveal correction terms that are
invisible in the continuous inequality, but their one-dimensional mechanisms
do not directly resolve the growing multiplicity present in \eqref{eq:Cintro}.

The high-dimensional problem was initiated by Gupta, who proved that
$C_\ell(N)$ has order $N^\ell$ for every fixed $\ell$ \cite{Gupta2023}.  Thus
dimension enters with exponent $\ell$ on the lattice, instead of $2\ell$ in
the Euclidean problem.  Very recently, Gupta \cite{Gupta2026} and Huang--Ye
\cite{HuangYe2026} independently determined the leading coefficient:
\begin{equation}\label{eq:leading-intro}
 \lim_{N\to\infty}\frac{C_\ell(N)}{N^\ell}=2^\ell.
\end{equation}
Their proofs reduce the lattice inequality to singular weighted estimates on
the flat torus.  Gupta uses a ground-state representation and weighted
Bochner identities, while Huang--Ye combine first- and second-order weighted
identities with dimension-uniform concentration.  The value $2^\ell$ is also
the leading Rayleigh quotient of functions supported on the first shell
\[
 S_1=\{\pm e_j:1\leq j\leq N\}.
\]

Both recent papers use the same full-lattice convention as
\eqref{eq:Cintro}: the power of the lattice Laplacian is formed on
$\ell^2(\Z^N)$, and the restriction $u(0)=0$ is imposed only on the
admissible functions.  Thus their optimal constants coincide with
$C_\ell(N)$ as defined here.  This convention is different from first
restricting the Laplacian to the punctured lattice and then taking its
$\ell$th power; the latter removes paths that visit the origin at an
intermediate time.

The limit \eqref{eq:leading-intro} leaves a genuinely degenerate problem at the
next scale.  At leading order, all $2N$ nearest neighbours of the origin have
the same energy.  Consequently, the leading asymptotic does not select a
unique approximate ground state: after conjugation, the leading operator is
scalar on the entire space supported on $S_1$.  Determining the first
correction requires resolving this growing degeneracy, identifying the
symmetry component that contains the spectral bottom, and quantifying its
coupling to the rest of the lattice.  At the following order, several distinct
orbit types interact.  There is an additional feature for $\ell\geq2$:
although admissible functions vanish at the origin, paths contributing to
matrix elements of the full-lattice power $L_N^\ell$ may pass through the
origin at intermediate times.  This contribution is already visible at order
$N^{\ell-1}$.

The main result resolves this structure through two further orders.  To our
knowledge, it is the first asymptotic expansion beyond the leading coefficient
for the constants \eqref{eq:Cintro}.

\begin{theorem}\label{thm:main-intro}
For every fixed integer $\ell\geq1$, as $N\to\infty$,
\begin{equation}\label{eq:main-intro}
 C_\ell(N)=2^\ell N^\ell+\gamma_\ell N^{\ell-1}
 +\delta_\ell N^{\ell-2}+O_\ell(N^{\ell-3}),
\end{equation}
where
\begin{equation}\label{eq:gamma-intro}
 \gamma_\ell=3\ell(\ell-1)2^{\ell-2}
 -\frac{\ell^2 2^{2\ell}}{2^\ell-1},
\end{equation}
and
\begin{align}
 \delta_\ell=2^\ell\bigg[&
 -\frac34\binom{\ell}{2}+\frac{15}{4}\binom{\ell}{4}
 +\frac{2^\ell\ell\bigl(\ell-6\binom{\ell}{3}\bigr)}{2^\ell-1}
 \nonumber\\
 &+\frac{2^\ell\ell^2\binom{\ell}{2}(5\cdot2^\ell-3)}
 {2(2^\ell-1)^2}
 +\frac{2^{2\ell}\ell^4}{(2^\ell-1)^3}
 \nonumber\\
 &-\frac{3^{\ell+1}}{2(3^\ell-1)}
 \left(\binom{\ell}{2}
 -\frac{2^\ell\ell^2}{2^\ell-1}\right)^2
 -\frac{2^{2\ell}\ell^2}{4(2^{2\ell}-1)}\bigg].
 \label{eq:delta}
\end{align}
Equivalently, there are constants $N_0(\ell)$ and $A_\ell$ such that the
absolute value of the remainder in \eqref{eq:main-intro} is at most
$A_\ell N^{\ell-3}$ for every $N\geq N_0(\ell)$.
\end{theorem}

The two terms in \eqref{eq:gamma-intro} have different spectral origins.  The
positive term comes from two-step returns to the first shell, including the
rank-one family of paths through the origin.  The negative term is the
second-order energy shift created by the leading coupling to the orbit
$(1,1)$.  At the next scale, the orbits $(1,1,1)$ and $(2)$ enter, together
with the next correction to the $(1,1)$ interaction.  Their combined effect is
the closed expression \eqref{eq:delta}.  For example, Theorem
\ref{thm:main-intro} gives
\[
 C_1(N)=2N-4-\frac{20}{3N}+O(N^{-2}).
\]

We now describe the proof at a conceptual level.  Write $R$ for multiplication
by $|n|$.  The quotient \eqref{eq:Cintro} becomes the spectral bottom
\[
 C_\ell(N)=\inf\sigma(K_{\ell,N}),\qquad
 K_{\ell,N}=R^\ell L_N^\ell R^\ell,
\]
on the punctured lattice.  The signed permutation group
$\BN=(\Z_2)^N\rtimes S_N$ preserves the operator and separates the
$2N$-dimensional first-shell space into three natural symmetry components.
This reduction identifies the component that can contain the lowest
eigenvalue.  Within the invariant component, lattice points are grouped by the
absolute values of their non-zero coordinates, so normalized orbit indicators
turn the low-energy problem into a weighted graph indexed by integer
partitions.  The relevant representation-theoretic background is recalled
from \cite{GeckPfeiffer}.

Exact path counting then shows that the four orbits
\[
 (1),\qquad(1,1),\qquad(1,1,1),\qquad(2)
\]
determine the eigenvalue through order $N^{\ell-2}$.  The resulting
four-dimensional effective operator produces both coefficients in
Theorem~\ref{thm:main-intro}.  This computation alone, however, gives only a
candidate expansion: one must still prove that the discarded lattice orbits
cannot change the bottom of the full spectrum at the stated order.

The principal analytic issue is therefore the control of the complement of
the four-orbit space.  Fourier transformation turns a shell restriction into
a lower bound on the Fourier level.  The singular torus weights needed for the
Hardy inequality mix those levels, so the ordinary Poincar\'e estimate loses
the required threshold.  We prove a weighted fixed-gap estimate that retains
it up to a relative $O(N^{-1/2})$ error.  This confines the low-energy
spectrum and supplies the resolvent bounds needed for a rigorous reduction.

Two Feshbach--Schur reductions, in the sense of
\cite{BachChenFrohlichSigal,DussonSigalStamm}, convert these bounds into scalar
spectral equations and separate the invariant eigenvalue branch from the
other first-shell components.  Finally, an orbit-path distance controls the
residual of the four-orbit eigenvector in the full invariant sector.  Temple's
inequality \cite{Temple} transfers the effective eigenvalue expansion to the
full operator with the error claimed in \eqref{eq:main-intro}.

The proof of Theorem~\ref{thm:main-intro} is completed in Section~7.  Section~2
constructs the self-adjoint realization, Sections~3 and~4 develop the orbit
algebra and the coefficient calculation, Section~5 resolves the first-shell
cluster, and Section~6 proves the fixed Fourier gap.  The paper concludes by
recording the full-shell path bounds used in the reductions and the exact
orbit moments underlying the effective matrix.

\section{Operator framework and conjugation}

Let
\[
 (L_Nu)(n)=2Nu(n)-\sum_{|m-n|=1}u(m),\qquad n\in\Z^N,
\]
be the non-negative combinatorial Laplacian.  For $\ell\in\N$, use the convention
\[
 (\nabla_j u)(n)=u(n+e_j)-u(n),\qquad 1\leq j\leq N,
 \qquad \nabla u=(\nabla_1u,\ldots,\nabla_Nu).
\]
With this normalization,
$\sum_{n,j}|\nabla_j u(n)|^2=\ip{u}{L_Nu}$.  We then set
\[
 D^\ell u=
 \begin{cases}
  L_N^{\ell/2}u,&\ell\text{ even},\\
  \nabla L_N^{(\ell-1)/2}u,&\ell\text{ odd},
 \end{cases}
\]
so that
\[
 \sum_{n\in\Z^N}|D^\ell u(n)|^2=\ip{u}{L_N^\ell u}.
\]
For finitely supported $u$ satisfying $u(0)=0$, define
\[
 C_\ell(N):=\inf_{u\ne0}
 \frac{\ip{u}{L_N^\ell u}}
 {\displaystyle\sum_{n\ne0}\frac{|u(n)|^2}{|n|^{2\ell}}}.
\]
All Hilbert spaces and test-function spaces are taken over $\mathbb C$; the
same optimal constant is obtained by restricting to real-valued functions.

Put $r(n)=|n|$ and let $R$ be multiplication by $r$.  Writing $u=R^\ell v$
gives
\[
 \sum_{n\ne0}\frac{|u(n)|^2}{|n|^{2\ell}}=\norm{v}_2^2.
\]
Let
\[
 \cH_{0,N}=\ell^2(\Z^N\setminus\{0\};\mathbb C)
\]
and identify it with the subspace of $\ell^2(\Z^N)$ obtained by extension by
zero at the origin.  On $C_c(\Z^N\setminus\{0\})$ define the non-negative
quadratic form
\begin{equation}\label{eq:form}
 \mathfrak k_{\ell,N}[v]
 =\ip{R^\ell v}{L_N^\ell R^\ell v}.
\end{equation}
The formal matrix
\[
 S_{\ell,N}v=R^\ell L_N^\ell R^\ell v,
 \qquad v\in C_c(\Z^N\setminus\{0\}),
\]
is a densely defined non-negative symmetric operator of finite propagation.
Its quadratic form is closable by \cite[Corollary VI.1.28]{Kato}.  The closure
is represented by a unique non-negative self-adjoint operator according to the
first representation theorem \cite[Theorem VI.2.6]{Kato}; this operator is the
Friedrichs extension of $S_{\ell,N}$ \cite[Theorem VI.2.11]{Kato}.  We write
$\overline{\mathfrak k}_{\ell,N}$ for that closure, with form domain
$\mathcal Q_{\ell,N}$, and denote by $K_{\ell,N}$ the unique non-negative
self-adjoint operator represented by it.  Thus the matrix identity
$K_{\ell,N}=R^\ell L_N^\ell R^\ell$ is used only on the core
$C_c(\Z^N\setminus\{0\})$; the operator domain of $K_{\ell,N}$ is determined
by the first representation theorem.  The variational characterization of the
Friedrichs extension gives
\begin{equation}\label{eq:Kdef}
 C_\ell(N)=\inf\sigma(K_{\ell,N})
 =\inf_{\substack{v\in C_c(\Z^N\setminus\{0\})\\v\ne0}}
 \frac{\mathfrak k_{\ell,N}[v]}{\norm v_2^2}.
\end{equation}
For $g\in\BN$, let $U_g$ be the unitary induced by the corresponding signed
coordinate permutation.  Each $U_g$ preserves
$C_c(\Z^N\setminus\{0\})$ and satisfies
$\mathfrak k_{\ell,N}[U_gv]=\mathfrak k_{\ell,N}[v]$ on this core.
It therefore preserves the closed form
$\overline{\mathfrak k}_{\ell,N}$ and its form domain.  Uniqueness in the
first representation theorem then gives
$U_gK_{\ell,N}=K_{\ell,N}U_g$.  Hence every isotypic component of the
signed-permutation action reduces the Friedrichs operator, not merely its
formal matrix on finitely supported functions.

Let $P$ be the projection onto any finite union of coordinate shells used
below.  Then $P$ maps the form core into a finite-dimensional subspace of that
core.  If $v_j$ converges to $v$ in the form norm, $Pv_j$ converges to $Pv$ in
$\ell^2$; on the finite-dimensional range of $P$ the form norm and the
$\ell^2$ norm are equivalent.  Hence $Pv_j$ converges to $Pv$ in the form
norm.  Thus both $P$ and $I-P$ preserve $\mathcal Q_{\ell,N}$.  Moreover,
$P\cH_{0,N}\subset\operatorname{Dom}(K_{\ell,N})$, because the formal matrix
maps a finite shell into a finite set, and the map
$(I-P)K_{\ell,N}P$ is bounded.  These facts justify the form compressions and
the finite-shell off-diagonal blocks used in Section~5.  The restriction to
$\cH_{0,N}$ also excludes the vector $\delta_0$,
which would otherwise be killed by both factors $R^\ell$.
Here $L_N^\ell$ is the power of the Laplacian on the whole lattice.  This detail is
essential for $\ell\geq2$: paths contributing to $L_N^\ell$ may pass through the
origin even though $u(0)=0$.

The distinction between endpoints and intermediate vertices is recorded in
the following observation.

\begin{remark}\label{rem:origin}
The condition $u(0)=0$ removes the origin as an endpoint, but it does not remove
lattice paths in the matrix elements of $L_N^\ell$ whose intermediate vertices
visit the origin.  In particular, the compression of $A_N^2$ to the first shell
contains the rank-one contribution denoted below by $\mathcal J_N$.  This
contribution enters the diagonal moment \eqref{eq:A2moment} and hence the term
$3\ell(\ell-1)2^{\ell-2}$ in $\gamma_\ell$.  Thus the full-lattice power must
be formed before the admissibility condition at the origin is imposed.
\end{remark}

\section{Hyperoctahedral orbit algebra}

Let $\BN=(\Z_2)^N\rtimes S_N$ be the signed permutation group, also known as
the hyperoctahedral group or the finite Coxeter group of type $B_N$; see
\cite{GeckPfeiffer} for its representation theory.  By the closed-form
argument in Section~2, it acts by signed coordinate permutations and commutes
with $K_{\ell,N}$.  A non-zero orbit is represented
by a partition
\[
 \lambda=(a_1,\ldots,a_k),\qquad a_1\geq\cdots\geq a_k\geq1.
\]
Write
\[
 s_\lambda=\sum_{j=1}^k a_j^2,
 \qquad
 \psi_\lambda=|\cO_\lambda|^{-1/2}\1_{\cO_\lambda}.
\]
We also use the empty partition $\varnothing$ for the orbit $\{0\}$ whenever
the origin occurs as an intermediate vertex of an adjacency path.  It is
never an endpoint of a matrix element of $K_{\ell,N}$, because the endpoint
factor $R^\ell$ vanishes there.  Part~\textnormal{(i)} of the next lemma
extends to $\lambda=\varnothing$ and gives
$\ip{A_N\psi_\varnothing}{\psi_{(1)}}=\sqrt{2N}$.
Let $A_N=2NI-L_N$ be the adjacency operator.

The first orbit identity gives all matrix entries of $A_N$ needed below.

\begin{lemma}\label{lem:adjacency}
In the normalized orbit basis, the only non-zero off-diagonal coefficients of
$A_N$ are the following.
\begin{enumerate}[label=\textnormal{(\roman*)}]
\item If $\mu=\lambda\cup\{1\}$, then
\[
 \ip{A_N\psi_\lambda}{\psi_\mu}
 =\sqrt{2(N-k)(m_1(\lambda)+1)}.
\]
\item If $\mu$ is obtained by replacing one occurrence of $a$ by $a+1$, then
\[
 \ip{A_N\psi_\lambda}{\psi_\mu}
 =\sqrt{m_a(\lambda)(m_{a+1}(\lambda)+1)}.
\]
\end{enumerate}
The reverse coefficients follow by symmetry.
\end{lemma}

\begin{proof}
For adjacent orbits, let $d_{\lambda\mu}$ be the number of neighbours in
$\cO_\mu$ of one point of $\cO_\lambda$.  This number is independent of the
chosen point because $\BN$ acts transitively on each orbit and preserves
adjacency.  Counting the edges between the two orbits from both endpoints gives
\[
 |\cO_\lambda|d_{\lambda\mu}
 =|\cO_\mu|d_{\mu\lambda}.
\]
On the other hand, the definition of the normalized orbit vectors yields
\[
 \ip{A_N\psi_\lambda}{\psi_\mu}
 =\frac{|\cO_\lambda|d_{\lambda\mu}}
 {\sqrt{|\cO_\lambda||\cO_\mu|}}
 =\sqrt{d_{\lambda\mu}d_{\mu\lambda}}.
\]

Suppose first that $\mu=\lambda\cup\{1\}$ and that $\lambda$ has $k$ parts.
Starting from a point of $\cO_\lambda$, one may choose any of the $N-k$ zero
coordinates and either sign, so $d_{\lambda\mu}=2(N-k)$.  Conversely, a point
of $\cO_\mu$ has $m_1(\lambda)+1$ coordinates of absolute value one that can
be moved to zero, and hence
$d_{\mu\lambda}=m_1(\lambda)+1$.  Their product gives part~(i).

For part~(ii), an outward move is obtained by choosing one of the
$m_a(\lambda)$ coordinates of absolute value $a$ and increasing its absolute
value by one without changing its sign.  Thus
$d_{\lambda\mu}=m_a(\lambda)$.  The reverse move may be made at any of the
$m_{a+1}(\lambda)+1$ coordinates of absolute value $a+1$ in $\mu$.  Therefore
$d_{\mu\lambda}=m_{a+1}(\lambda)+1$, which proves the second formula.  Since
$A_N$ is self-adjoint, interchanging $\lambda$ and $\mu$ gives all reverse
coefficients.
\end{proof}

An edge in the orbit graph is called \emph{dimensional} if it adds or removes a
part equal to $1$ and \emph{internal} otherwise.  For a path $\pi$, write
$j(\pi)$ for its length and $a(\pi)$ for the number of dimensional edges, and
put
\begin{equation}\label{eq:weighted-distance}
 d_*(\lambda,\mu)
 =\min_\pi\left(j(\pi)-\frac{a(\pi)}2\right),
\end{equation}
where the minimum is over paths joining $\lambda$ and $\mu$.

The weighted distance converts the orbit combinatorics into uniform powers of
$N$.

\begin{lemma}\label{lem:pathbound}
Fix $J$ and two finite sets of partitions.  Uniformly for $0\leq j\leq J$ and
for all partitions $\lambda,\mu$ in those sets or reachable from them in at
most $J$ steps,
\begin{equation}\label{eq:Ajbound}
 \left|\ip{A_N^j\psi_\lambda}{\psi_\mu}\right|
 \leq C_J N^{a_j(\lambda,\mu)/2},
\end{equation}
where $a_j(\lambda,\mu)$ is the largest number of dimensional edges among
length-$j$ paths from $\lambda$ to $\mu$; the left side is zero if no such path
exists.  Consequently, for fixed $\ell$ and bounded $s_\lambda,s_\mu$,
\begin{equation}\label{eq:Kpathbound}
 \left|\ip{K_{\ell,N}\psi_\lambda}{\psi_\mu}\right|
 \leq C_\ell N^{\ell-d_*(\lambda,\mu)}.
\end{equation}
\end{lemma}

\begin{proof}
For the fixed initial sets and fixed $J$, only finitely many partitions can be
reached in at most $J$ steps.  Their lengths and multiplicities are therefore
bounded by constants depending on $J$ but not on $N$.  Lemma
\ref{lem:adjacency} then shows that a dimensional edge has coefficient at most
$C_JN^{1/2}$, while an internal edge has coefficient at most $C_J$.

Expanding the matrix product $A_N^j$ in the orbit basis expresses
$\ip{A_N^j\psi_\lambda}{\psi_\mu}$ as a sum over all length-$j$ orbit paths
from $\lambda$ to $\mu$.  A path with $a$ dimensional edges is bounded by
$C_JN^{a/2}$.  The number of such paths is bounded solely in terms of $J$,
because every vertex in the relevant finite orbit graph has bounded degree.
Taking the largest possible value of $a$ proves \eqref{eq:Ajbound}; if no path
exists, every summand is absent and the matrix entry is zero.

Since $R^\ell$ is constant on $\cO_\lambda$ with value
$s_\lambda^{\ell/2}$, we have
\[
 \ip{K_{\ell,N}\psi_\lambda}{\psi_\mu}
 =s_\lambda^{\ell/2}s_\mu^{\ell/2}
 \sum_{j=0}^\ell(-1)^j\binom\ell j(2N)^{\ell-j}
 \ip{A_N^j\psi_\lambda}{\psi_\mu}.
\]
For a contributing path of length $j$ with $a$ dimensional edges, the
corresponding term is $O_\ell(N^{\ell-j+a/2})$.  By definition,
$j-a/2\geq d_*(\lambda,\mu)$ for every such path.  Each term is therefore
$O_\ell(N^{\ell-d_*(\lambda,\mu)})$, and only finitely many values of $j$ and
paths occur.  The bounded endpoint factors are absorbed into the constant,
which proves \eqref{eq:Kpathbound}.
\end{proof}

We use the four orbit types
\begin{equation}\label{eq:fourorbits}
 (1),\qquad(1,1),\qquad(1,1,1),\qquad(2).
\end{equation}
Let $M_{\ell,N}$ be the compression of $K_{\ell,N}$ to their span and let
$\mu_{\ell,N}=\lambda_{\min}(M_{\ell,N})$.

The four-orbit compression already determines the leading term and its first
correction.  Its third coefficient is stated and proved separately in the
next section, after all matrix entries needed at that order have been
collected.

\begin{proposition}\label{prop:ritz}
For every fixed $\ell\geq1$,
\begin{equation}\label{eq:ritzleading}
 \mu_{\ell,N}
 =2^\ell N^\ell+\gamma_\ell N^{\ell-1}
  +O_\ell(N^{\ell-2}),
\end{equation}
where
\begin{equation}\label{eq:gamma}
 \gamma_\ell
 =3\ell(\ell-1)2^{\ell-2}
  -\frac{\ell^2 2^{2\ell}}{2^\ell-1}.
\end{equation}
In particular,
\[
 \gamma_1=-4,\qquad
 \gamma_2=-\frac{46}{3},\qquad
 \gamma_3=-\frac{324}{7},\qquad
 \gamma_4=-\frac{1936}{15}.
\]
The variational principle therefore gives the upper bound
\begin{equation}\label{eq:upper}
 C_\ell(N)\leq
 2^\ell N^\ell+\gamma_\ell N^{\ell-1}+O_\ell(N^{\ell-2}).
\end{equation}
\end{proposition}

\begin{proof}
Since $L_N=2NI-A_N$,
\begin{equation}\label{eq:binomial}
 L_N^\ell=\sum_{j=0}^\ell
 (-1)^j\binom{\ell}{j}(2N)^{\ell-j}A_N^j.
\end{equation}
Because $R^\ell$ is constant on every orbit, the matrix entries of
$M_{\ell,N}$ are obtained from \eqref{eq:binomial} by multiplying the
$(\lambda,\mu)$ entry by
$s_\lambda^{\ell/2}s_\mu^{\ell/2}$.  We first isolate the entries that can
contribute at order $N^{\ell-1}$.

The first orbit satisfies
\begin{equation}\label{eq:A2moment}
 \ip{A_N^2\psi_{(1)}}{\psi_{(1)}}
 =2N+4(N-1)+1=6N-3.
\end{equation}
The three terms correspond respectively to paths through the origin, through
the orbit $(1,1)$, and through the orbit $(2)$.  More explicitly, the
intermediate vertex $0$ contributes $2N$, the normalized return through
$(1,1)$ contributes $4(N-1)$, and the two axial neighbours $\pm2e_j$ together
contribute $1$.  At the full first-shell level these three contributions are
$\mathcal J_N$, $\mathcal B_N$, and $P_1$, respectively; thus the identity
used later is
$P_1A_N^2P_1=\mathcal J_N+\mathcal B_N+P_1$.
Moreover, Lemma~\ref{lem:adjacency} gives
\[
 \ip{A_N\psi_{(1)}}{\psi_{(1,1)}}=2\sqrt{N-1}.
\]
Since $s_{(1)}=1$ and $s_{(1,1)}=2$, the $j=1$ term in
\eqref{eq:binomial} acquires the factor $2^{\ell/2}$.  The $j=2$ term in the
first diagonal entry and the leading $j=0$ terms therefore give
\begin{align*}
 (M_{\ell,N})_{11}
 &=(2N)^\ell
   +6\binom{\ell}{2}(2N)^{\ell-2}N
   +O_\ell(N^{\ell-2}),\\
 (M_{\ell,N})_{12}
 &=-\ell(2N)^{\ell-1}2^{\ell/2+1}\sqrt{N}
   +O_\ell(N^{\ell-3/2}),\\
 (M_{\ell,N})_{22}
 &=2^\ell(2N)^\ell+O_\ell(N^{\ell-1}).
\end{align*}
Here $6N-3$ may be replaced by $6N$ at this stage because the constant part
contributes only $O_\ell(N^{\ell-2})$.  Likewise,
$\sqrt{N-1}=\sqrt N+O(N^{-1/2})$, and the resulting correction to the coupling
is $O_\ell(N^{\ell-3/2})$.

We next compare the scales of the remaining two coordinates.  Their diagonal
entries are separated from $(2N)^\ell$ by a quantity of order $N^\ell$.
Lemma~\ref{lem:pathbound} gives couplings of size at most
$O_\ell(N^{\ell-1})$ from the first orbit, so their Schur-complement
contribution is $O_\ell(N^{\ell-2})$.  Thus the coefficient of
$N^{\ell-1}$ is determined by the leading $2\times2$ block on
$(1),(1,1)$.

Let $a_N,d_N$ denote its diagonal entries and $c_N$ its off-diagonal entry.
Since $d_N-a_N=(2^\ell-1)(2N)^\ell+O_\ell(N^{\ell-1})$, the lower eigenvalue
may be expanded as
\[
 \frac{a_N+d_N-\sqrt{(d_N-a_N)^2+4c_N^2}}2
 =a_N-\frac{c_N^2}{d_N-a_N}+O_\ell(N^{\ell-2}).
\]
Indeed, $c_N/(d_N-a_N)=O_\ell(N^{-1/2})$, so the omitted terms in the square
root begin with $c_N^4/(d_N-a_N)^3=O_\ell(N^{\ell-2})$.  Substitution of the
three entries above gives
\begin{align*}
 \mu_{\ell,N}
 &=(2N)^\ell
 +3\ell(\ell-1)2^{\ell-2}N^{\ell-1}\\
 &\quad
 -\frac{\ell^2 2^{3\ell}N^{2\ell-1}}
 {(2N)^\ell(2^\ell-1)}
 +O_\ell(N^{\ell-2}),
\end{align*}
which is \eqref{eq:gamma}.  Determining the coefficient of $N^{\ell-2}$
requires retaining every selected-orbit path of weighted cost at most two.
That finite calculation is carried out in Proposition~\ref{prop:delta} and
Appendix~\ref{app:matrix}; it is deliberately not used in the present
statement.

Finally, the four orbit vectors are finitely supported and belong to the form
core of $K_{\ell,N}$.  The Rayleigh--Ritz principle applied to their span gives
$C_\ell(N)\leq\mu_{\ell,N}$.  Equation~\eqref{eq:ritzleading} proves
\eqref{eq:upper}.
\end{proof}

\section{Derivation of the third coefficient}

The next proposition derives the closed expression \eqref{eq:delta} directly
from the four-orbit matrix.

\begin{proposition}\label{prop:delta}
For every fixed $\ell\geq1$,
\begin{equation}\label{eq:ritzexp}
 \mu_{\ell,N}
 =2^\ell N^\ell+\gamma_\ell N^{\ell-1}
  +\delta_\ell N^{\ell-2}+O_\ell(N^{\ell-3}),
\end{equation}
where $\delta_\ell$ is given by \eqref{eq:delta}.  Its first values are
\[
 \delta_1=-\frac{20}{3},\qquad
 \delta_2=-\frac{7729}{540},\qquad
 \delta_3=-\frac{30020}{4459},\qquad
 \delta_4=\frac{5639018}{57375}.
\]
\end{proposition}

\begin{proof}
Put $t=N^{-1/2}$ and write
\[
 p_2=2^\ell,\qquad p_3=3^\ell,
 \qquad c_j=\binom{\ell}{j}.
\]
We scale the four-orbit matrix by $(2N)^{-\ell}$ and denote the resulting
matrix by $\widehat M(t)$.  The exact moments in
Appendix~\ref{app:matrix}, after expansion in powers of $t$, give
\begin{align*}
 \widehat M_{00}
 &=1+\frac32c_2t^2+
   \left(-\frac34c_2+\frac{15}{4}c_4\right)t^4+O_\ell(t^6),\\
 \widehat M_{01}
 &=-\sqrt {p_2}\,\ell t+
   \sqrt {p_2}\left(\frac\ell2-3c_3\right)t^3+O_\ell(t^5),\\
 \widehat M_{02}
 &=\frac{\sqrt{6p_3}}2c_2t^2+O_\ell(t^4),\\
 \widehat M_{03}
 &=-\frac{p_2\ell}{2}t^2+O_\ell(t^4),\\
 \widehat M_{11}
 &=p_2+\frac52p_2c_2t^2+O_\ell(t^4),\\
 \widehat M_{12}
 &=-\frac\ell2\sqrt{6p_2p_3}\,t+O_\ell(t^3),\\
 \widehat M_{22}&=p_3+O_\ell(t^2),\\
 \widehat M_{33}&=p_2^2+O_\ell(t^2).
\end{align*}

For later use, abbreviate the displayed coefficients by
\begin{align*}
 a_0&=\frac32c_2,
 &a_4&=-\frac34c_2+\frac{15}{4}c_4,\\
 A&=\sqrt {p_2}\,\ell,
 &B&=\sqrt {p_2}\left(\frac\ell2-3c_3\right),\\
 C&=\frac{\sqrt{6p_3}}2c_2,
 &E&=-\frac{p_2\ell}{2},\\
 \kappa&=\frac52p_2c_2,
 &G&=\frac\ell2\sqrt{6p_2p_3}.
\end{align*}
At $t=0$ the matrix is diagonal with entries
$1,p_2,p_3,p_2^2$.  Since $p_2,p_3,p_2^2>1$, the eigenvalue $1$ is simple and
is separated from the other three eigenvalues.  Analytic perturbation theory
therefore produces a unique eigenvalue branch issuing from $1$ and an
eigenvector whose first component may be normalized to one.  The orbit-path
costs in Appendix~\ref{app:matrix} determine the first possible powers of
$t$ in the other components.  We may consequently write
\[
 \lambda=1+gt^2+ht^4+O_\ell(t^6),
 \qquad
 v=(1,b_1t+b_3t^3,e_2t^2,d_2t^2)^T+O_\ell(t^4)
\]
in $\widehat M(t)v(t)=\lambda(t)v(t)$.

We now compare the four coordinates.  In the second coordinate, the
coefficients of $t$ and $t^3$ give
\begin{align}
 b_1&=\frac{A}{p_2-1},\label{eq:b1}\\
 b_3&=\frac{-B-(\kappa-g)b_1+Ge_2}{p_2-1}.
 \label{eq:b3}
\end{align}
The coefficients of $t^2$ in the third and fourth coordinates give
\begin{align}
 e_2&=\frac{Gb_1-C}{p_3-1},\label{eq:c2star}\\
 d_2&=-\frac{E}{p_2^2-1}.
 \label{eq:d2}
\end{align}
Finally, the coefficients of $t^2$ and $t^4$ in the first coordinate are
\begin{align}
 g&=a_0-Ab_1,\label{eq:g}\\
 h&=a_4-Ab_3+Bb_1+Ce_2+Ed_2.
 \label{eq:h}
\end{align}
Equations \eqref{eq:b1} and \eqref{eq:g} show directly that
\[
 2^\ell g
 =3\ell(\ell-1)2^{\ell-2}
 -\frac{\ell^22^{2\ell}}{2^\ell-1}
 =\gamma_\ell.
\]

It remains to eliminate the auxiliary coefficients from $h$.  Substituting
\eqref{eq:b1}--\eqref{eq:g} into \eqref{eq:h}, and grouping the two terms
containing $e_2$, yields
\begin{equation}\label{eq:hclosed}
 h=a_4+\frac{2AB}{p_2-1}
 +\frac{A^2(\kappa-a_0)}{(p_2-1)^2}
 +\frac{A^4}{(p_2-1)^3}
 -\frac{\bigl(C-AG/(p_2-1)\bigr)^2}{p_3-1}
 -\frac{E^2}{p_2^2-1}.
\end{equation}
Indeed, the coefficient of $e_2$ in \eqref{eq:h} becomes
$C-AG/(p_2-1)$, whereas
$e_2=-(C-AG/(p_2-1))/(p_3-1)$; this gives the square in
\eqref{eq:hclosed}.  The remaining substitutions are
\begin{align*}
 2AB&=p_2\ell(\ell-6c_3),\\
 A^2(\kappa-a_0)&=\frac12p_2\ell^2c_2(5p_2-3),\\
 C-\frac{AG}{p_2-1}
 &=\frac{\sqrt{6p_3}}2
 \left(c_2-\frac{p_2\ell^2}{p_2-1}\right),\\
 E^2&=\frac14p_2^2\ell^2.
\end{align*}
Inserting these identities into \eqref{eq:hclosed} and multiplying by
$p_2=2^\ell$ gives exactly \eqref{eq:delta}.  Finally,
$(2N)^\ell t^4=2^\ell N^{\ell-2}$ and
$(2N)^\ell t^6=O_\ell(N^{\ell-3})$, so the eigenvalue expansion has the
claimed coefficient and remainder.  Substitution of $\ell=1,2,3,4$ in
\eqref{eq:delta} gives the four displayed values.
\end{proof}

\section{Selection of the invariant first-shell sector}

For $\ell=1$, positivity of the Dirichlet form allows a nonlinear symmetrization.
For $\ell\geq2$, this is no longer available: $L_N^\ell$ is not a Markov generator.
One must instead compare the irreducible components of the first-shell cluster.

Two finite-dimensional reductions have distinct roles.  The four-orbit
compression of Sections~3--4 computes the coefficients inside the invariant
sector.  The full first-shell Feshbach reduction below compares all
representations meeting $S_1$ and excludes competing branches.  Neither
reduction substitutes for the other.

Let $\cE_N$ be the $2N$-dimensional space supported on
$S_1=\{\pm e_j:1\leq j\leq N\}$.  The action of the signed permutation group
$\BN$ gives the orthogonal decomposition
\begin{equation}\label{eq:decomposition}
 \cE_N=\cE_N^{\mathrm{triv}}
 \oplus\cE_N^{\mathrm{ev},0}
 \oplus\cE_N^{\mathrm{odd}},
\end{equation}
where, writing $f_j^\pm=f(\pm e_j)$,
\begin{align*}
 \cE_N^{\mathrm{triv}}
 &=\{f:f_j^+=f_j^-=a\text{ for all }j\},\\
 \cE_N^{\mathrm{ev},0}
 &=\{f:f_j^+=f_j^-=a_j\text{ and }\textstyle\sum_{j=1}^N a_j=0\},\\
 \cE_N^{\mathrm{odd}}
 &=\{f:f_j^+=-f_j^-\text{ for all }j\}.
\end{align*}
The three summands have dimensions $1$, $N-1$, and $N$, respectively.  Their
irreducibility can be seen directly.  On the even subspace the coordinate sign
changes act trivially, and the action reduces to the permutation representation
of $S_N$ on $\mathbb C^N$.  Its constant line and zero-sum subspace are
irreducible: if a non-zero invariant subspace of the zero-sum component
contains $a$ with $a_i\ne a_j$, then
$a-(ij)a=(a_i-a_j)(e_i-e_j)$ belongs to it, and the permutations of
$e_i-e_j$ span the entire zero-sum space.  On the odd subspace a sign change
in the $j$th coordinate acts as the diagonal reflection
$a_j\mapsto-a_j$.  Hence $(I-s_j)a/2=a_je_j$ belongs to every invariant
subspace containing $a$; a non-zero coordinate followed by permutations
generates all of $\mathbb C^N$.  The odd component is therefore irreducible as
well.  Thus \eqref{eq:decomposition} is multiplicity-free, and every
$\BN$-equivariant self-adjoint operator on $\cE_N$ is scalar on each of the
three summands.

Let $\mathcal J_N$ be the matrix of two-step paths through the origin and let
$\mathcal B_N$ be the matrix of two-step paths through the orbit $(1,1)$.
On $S_1$,
\[
 \mathcal J_N(x,y)=1,
\]
while
\[
 \mathcal B_N(x,x)=2(N-1),\quad
 \mathcal B_N(x,-x)=0,\quad
 \mathcal B_N(\sigma e_i,\tau e_j)=1\quad(i\ne j).
\]
Their eigenvalues on the three components in \eqref{eq:decomposition} are
\begin{center}
\begin{tabular}{c|ccc}
 & $\cE_N^{\mathrm{triv}}$ & $\cE_N^{\mathrm{ev},0}$ & $\cE_N^{\mathrm{odd}}$\\ \hline
$\mathcal J_N$ & $2N$ & $0$ & $0$\\
$\mathcal B_N$ & $4(N-1)$ & $2N-4$ & $2N-2$
\end{tabular}
\end{center}

The matrix $\mathcal J_N$ is the all-ones matrix on $S_1$, so it has
eigenvalue $2N$ on the constant vector and vanishes on its orthogonal
complement.  The row sum of $\mathcal B_N$ is $4(N-1)$, which gives its
trivial eigenvalue.  If $f(\pm e_j)=a_j$ and $\sum_j a_j=0$, then
\[
 (\mathcal B_Nf)(e_i)
 =2(N-1)a_i+2\sum_{j\ne i}a_j=(2N-4)a_i.
\]
If instead $f(e_j)=-f(-e_j)$, the contributions from every pair
$\{e_j,-e_j\}$ with $j\ne i$ cancel, leaving
$(\mathcal B_Nf)(e_i)=2(N-1)f(e_i)$.  These computations prove all entries in
the displayed array.

Let $P_1$ be the projection onto $S_1$, let $P_{11}$ be the projection onto the
orbit $(1,1)$, and put
\[
 Q_1=I-P_1,\qquad Z_3=I-P_1-P_{11}.
\]
Notice that $Z_3$ also contains the axial orbit $(2)$: its Fourier index has
squared length $4$, not $2$.  Set $\Lambda_N=(2N)^\ell$.

We use the following projection form of the Feshbach--Schur map.  It is the
Schur complement of the $Q$-block and is isospectral at zero in the sense made
precise below; see \cite{BachChenFrohlichSigal} for a general formulation.

\begin{lemma}\label{lem:schur}
Let $H$ be a semibounded self-adjoint operator with closed form $\mathfrak h$,
and let $P$ be a finite-rank orthogonal projection such that
$P\cH\subset\operatorname{Dom}(H)$.  Put $Q=I-P$, assume that $P$ preserves
the form domain, and let $H_Q$ be the self-adjoint operator represented by the
restriction of $\mathfrak h$ to $Q\mathcal Q(H)$.  Let $z\in\mathbb R$.  If
$H_Q-z\geq\eta Q$ for some $\eta>0$, then
\[
 \mathcal F_P(H-z)
 =P(H-z)P-PHQ(H_Q-z)^{-1}QHP
\]
is self-adjoint on the finite-dimensional space $P\cH$; here $PHQ$ denotes
the bounded adjoint of $QHP$.  Moreover,
\[
 z\in\sigma(H)\quad\Longleftrightarrow\quad
 0\in\sigma\bigl(\mathcal F_P(H-z)\bigr).
\]
The dimensions of the corresponding kernels agree.  If $b_Q$ is a lower form
bound for $H_Q$, then the spectral projection of $H$ on $(-\infty,b_Q)$ has
dimension at most $\operatorname{rank}P$.  Its range is a finite sum of
eigenspaces, and every eigenvalue in this interval is detected by the Schur
complement.
\end{lemma}

\begin{proof}
Finite rank and $P\cH\subset\operatorname{Dom}(H)$ imply that
$A=PHP$ and $B=QHP$ are bounded on their finite-dimensional domain; the
adjoint of $B$ is the bounded map $PHQ$.  For
$p\in P\cH$ and $q\in Q\mathcal Q(H)$ the closed form decomposes as
\[
 \mathfrak h[p+q]=\ip{Ap}{p}+2\operatorname{Re}\ip{Bp}{q}
 +\mathfrak h_Q[q].
\]
The first representation theorem therefore identifies $H$ with the block
operator
\[
 \begin{pmatrix}A&B^*\\ B&H_Q\end{pmatrix},
 \qquad
 \operatorname{Dom}(H)=P\cH\oplus\operatorname{Dom}(H_Q).
\]
The assumption on $H_Q-z$ makes $D=H_Q-z$ invertible and gives
$\norm{D^{-1}}\leq\eta^{-1}$.  Since
$D^{-1}:Q\cH\to\operatorname{Dom}(H_Q)$ and
$B:P\cH\to Q\cH$ is bounded, one has
$D^{-1}B(P\cH)\subset\operatorname{Dom}(H_Q)$.  Thus the right triangular
factor below preserves $P\cH\oplus\operatorname{Dom}(H_Q)$.  As an identity from
$P\cH\oplus\operatorname{Dom}(H_Q)$ to $P\cH\oplus Q\cH$, block
multiplication yields
\[
 H-z=
 \begin{pmatrix}I&B^*D^{-1}\\0&I\end{pmatrix}
 \begin{pmatrix}\mathcal F_P(H-z)&0\\0&D\end{pmatrix}
 \begin{pmatrix}I&0\\D^{-1}B&I\end{pmatrix}.
\]
The triangular factors are bounded and have bounded inverses obtained by
changing the signs of their off-diagonal entries.  Consequently, $H-z$ is
invertible if and only if both diagonal factors are invertible.  Since $D$ is
already invertible, this is equivalent to invertibility of
$\mathcal F_P(H-z)$, proving the spectral equivalence.

The kernel correspondence can also be read explicitly.  If
$(H-z)(p+q)=0$ with $p=P(p+q)$ and $q=Q(p+q)$, then the $Q$-equation gives
$q=-D^{-1}Bp$; substitution into the $P$-equation gives
$\mathcal F_P(H-z)p=0$.  Conversely, any vector
$p\in\ker\mathcal F_P(H-z)$ produces
$p-D^{-1}Bp\in\ker(H-z)$.  These mutually inverse maps preserve the kernel
dimension.

For the final assertion, suppose that the spectral projection
$E_H((-\infty,b_Q))$ had dimension larger than $\operatorname{rank}P$.  Its
range would contain a non-zero vector $f$ with $Pf=0$.  The spectral theorem
would give $\mathfrak h[f]<b_Q\norm f^2$, whereas
$f\in Q\mathcal Q(H)$ and the definition of $b_Q$ gives
$\mathfrak h[f]\geq b_Q\norm f^2$.  This contradiction proves the rank bound.
A finite-dimensional spectral subspace is a finite sum of eigenspaces, and the
kernel correspondence detects each corresponding eigenvalue.
\end{proof}

The following estimates put the shell couplings needed in both Schur
reductions into operator-norm form.  Let
\[
 C_N=P_1A_NP_{11}:P_{11}\cH_{0,N}\longrightarrow P_1\cH_{0,N}.
\]
Thus $C_N$ is the incidence operator between $S_1$ and the orbit $(1,1)$.

\begin{lemma}\label{lem:shellblocks}
For every fixed $\ell$, as $N\to\infty$,
\begin{align}
 P_1K_{\ell,N}P_1
 &=\Lambda_NP_1+\binom\ell2(2N)^{\ell-2}
 (\mathcal J_N+\mathcal B_N)+O_\ell(N^{\ell-2}),
 \label{eq:blocksPKP}\\
 P_1K_{\ell,N}Q_1
 &=-\ell(2N)^{\ell-1}2^{\ell/2}C_N+T_{\ell,N},
 \qquad \norm{T_{\ell,N}}=O_\ell(N^{\ell-1}),
 \label{eq:blocksPKQ}\\
 P_{11}K_{\ell,N}P_{11}
 &=2^\ell\Lambda_NP_{11}+O_\ell(N^{\ell-1}),
 \label{eq:blocks1111}\\
 \norm{P_{11}K_{\ell,N}Z_3}
 &=O_\ell(N^{\ell-1/2}).
 \label{eq:blocks11Z}
\end{align}
All remainders are operator-norm bounds.  Moreover,
$C_NC_N^*=\mathcal B_N$ and $\norm{C_N}=2\sqrt{N-1}$.
\end{lemma}

\begin{proof}
The full-shell path estimate in Lemma~\ref{lem:blockpath} says that a
dimensional orbit edge has operator norm $O(N^{1/2})$, an internal edge has
norm $O(1)$, and the resulting bounds are uniform over all irreducible
components.  Applied to the finitely many orbit types reachable from $S_1$ or
$(1,1)$ in at most $\ell$ steps, it also absorbs the bounded endpoint factors
$R^\ell$.

We apply this count to
\begin{equation}\label{eq:binomial-block}
 K_{\ell,N}=\sum_{j=0}^\ell(-1)^j\binom\ell j
 (2N)^{\ell-j}R^\ell A_N^jR^\ell.
\end{equation}
On $S_1$ the $j=0$ term is $\Lambda_NP_1$.  Lattice bipartiteness removes
the odd return blocks, and
\[
 P_1A_N^2P_1=\mathcal J_N+\mathcal B_N+P_1.
\]
The last summand is the return through the axial orbit $(2)$; after
multiplication by $(2N)^{\ell-2}$ it is $O_\ell(N^{\ell-2})$.  A return path
of even length $j\geq4$ has at most $j$ dimensional edges and is therefore
bounded by
$O_\ell(N^{\ell-j+j/2})=O_\ell(N^{\ell-2})$.  This proves
\eqref{eq:blocksPKP}.

For the $P_1$--$Q_1$ block, the one-step path to $(1,1)$ gives exactly
\[
 -\ell(2N)^{\ell-1}2^{\ell/2}P_1A_NP_{11}.
\]
The other one-step endpoint is the axial orbit $(2)$; its incidence norm is
$1$, so its contribution is $O_\ell(N^{\ell-1})$.  Every term with
$j\geq2$ is at most
$O_\ell(N^{\ell-j+j/2})=O_\ell(N^{\ell-1})$, after summing the finitely many
reachable endpoint orbits.  This proves \eqref{eq:blocksPKQ}.

On the orbit $(1,1)$ the two endpoint factors contribute $2^\ell$, so the
$j=0$ term is $2^\ell\Lambda_NP_{11}$.  The first return term has length two,
norm $O(N)$, and coefficient $O(N^{\ell-2})$; all longer returns are smaller.
This proves \eqref{eq:blocks1111}.  Finally, a one-step path from $(1,1)$ into
$Z_3$ ends either in $(1,1,1)$, with norm $O(N^{1/2})$, or in $(2,1)$, with
norm $O(1)$.  Its coefficient is $O(N^{\ell-1})$.  Terms of length at least
two are $O_\ell(N^{\ell-1})$, and hence
\eqref{eq:blocks11Z} follows.

For $x,y\in S_1$, the matrix entry of $C_NC_N^*$ counts their common
$(1,1)$-neighbours, which is exactly the definition of $\mathcal B_N$.
The largest eigenvalue of $\mathcal B_N$ is $4(N-1)$, as computed above, so
$\norm{C_N}=2\sqrt{N-1}$.
\end{proof}

The next estimate supplies the two inverse blocks used in the effective
Hamiltonian.

\begin{lemma}\label{lem:resolvent}
Fix $C>0$ and assume $|z-\Lambda_N|\leq C N^{\ell-1}$.  For all sufficiently
large $N$, the operators represented by the compressed forms satisfy
\begin{align}
 Q_1(K_{\ell,N}-z)Q_1
 &\geq c_\ell N^\ell Q_1,\label{eq:Qgap}\\
 Z_3(K_{\ell,N}-z)Z_3
 &\geq c'_\ell N^\ell Z_3.\label{eq:Zgap}
\end{align}
Moreover, on $P_{11}\cH_{0,N}$,
\begin{equation}\label{eq:resolvent11}
 P_{11}\bigl(Q_1(K_{\ell,N}-z)Q_1\bigr)^{-1}P_{11}
 =\frac{P_{11}}{(2^\ell-1)\Lambda_N}
 +O_\ell(N^{-\ell-1})
\end{equation}
in operator norm.
\end{lemma}

\begin{proof}
Proposition~\ref{prop:fouriergap}, applied with $q=2$ and $q=3$, gives
\[
 Q_1K_{\ell,N}Q_1\geq2^\ell\Lambda_N
 (1-O_\ell(N^{-1/2}))Q_1,
 \qquad
 Z_3K_{\ell,N}Z_3\geq3^\ell\Lambda_N
 (1-O_\ell(N^{-1/2}))Z_3.
\]
Subtracting $z=\Lambda_N+O(N^{\ell-1})$ from the first inequality leaves
\[
 \bigl((2^\ell-1)\Lambda_N-O_\ell(N^{\ell-1/2})
 -O(N^{\ell-1})\bigr)Q_1,
\]
which is bounded below by $c_\ell N^\ell Q_1$ for large $N$.  The same
calculation with $3^\ell-1$ in place of $2^\ell-1$ gives
\eqref{eq:Zgap}.  This proves \eqref{eq:Qgap}--\eqref{eq:Zgap} uniformly in
the stated energy window.

We next decompose $Q_1=P_{11}+Z_3$.  Equations
\eqref{eq:blocks1111}--\eqref{eq:blocks11Z} give the two block estimates
needed for this decomposition.  Feshbach reduction with respect to
$P_{11}+Z_3$, justified by
Lemma~\ref{lem:schur} and \eqref{eq:Zgap}, therefore gives
\begin{align*}
 &P_{11}(K_{\ell,N}-z)P_{11}
 -P_{11}K_{\ell,N}Z_3
 \bigl(Z_3(K_{\ell,N}-z)Z_3\bigr)^{-1}
 Z_3K_{\ell,N}P_{11}\\
 &\hspace{35mm}
 =\bigl((2^\ell-1)\Lambda_N\bigr)P_{11}+O_\ell(N^{\ell-1}).
\end{align*}
To see the size of the error explicitly, the $P_{11}$ diagonal contributes
$O_\ell(N^{\ell-1})$ after replacing $z$ by $\Lambda_N$, while the Schur term
is bounded by
\[
 O_\ell(N^{\ell-1/2})^2\,O_\ell(N^{-\ell})
 =O_\ell(N^{\ell-1}).
\]
Denote the left side by $D_{11,N}(z)$.  The error estimate is uniform in the
stated $z$-window, and
\[
 \norm{\bigl((2^\ell-1)\Lambda_N\bigr)^{-1}
       (D_{11,N}(z)-(2^\ell-1)\Lambda_NP_{11})}
 =O_\ell(N^{-1}).
\]
The Neumann series therefore converges uniformly and yields
\[
 D_{11,N}(z)^{-1}
 =\frac{P_{11}}{(2^\ell-1)\Lambda_N}
 +O_\ell(N^{-\ell-1}).
\]
The $P_{11}$-to-$P_{11}$ block of
$\bigl(Q_1(K_{\ell,N}-z)Q_1\bigr)^{-1}$ is exactly
$D_{11,N}(z)^{-1}$ by the block inverse formula.  This proves
\eqref{eq:resolvent11}.  Notice that replacing
$2^\ell\Lambda_N-z$ by $(2^\ell-1)\Lambda_N$ changes the reciprocal by
\[
 O(N^{\ell-1})\,O(N^{-2\ell})=O(N^{-\ell-1}),
\]
which is contained in the stated error.
\end{proof}

We can now compare the three irreducible first-shell components at the first
non-trivial order.

\begin{proposition}\label{prop:splitting}
The energy-dependent effective Hamiltonian on $\cE_N$, at energies
$(2N)^\ell+O(N^{\ell-1})$, is
\begin{align}
 (2N)^\ell I
 +(2N)^{\ell-2}
 \bigg[&\binom\ell2(\mathcal J_N+\mathcal B_N)
 \nonumber\\[-1mm]
 &-\frac{\ell^2 2^\ell}{2^\ell-1}\mathcal B_N\bigg]
 +O_\ell(N^{\ell-3/2}).\label{eq:feshbach}
\end{align}
More precisely, for each fixed $C>0$ the remainder is uniform for all real
$z$ satisfying $|z-\Lambda_N|\leq C N^{\ell-1}$.
Hence the trivial component has correction $\gamma_\ell N^{\ell-1}$, while both
non-trivial components have the common leading correction
\begin{equation}\label{eq:gammaperp}
 \gamma_\ell^\perp
 =\ell(\ell-1)2^{\ell-2}
 -\frac{\ell^2 2^{2\ell-1}}{2^\ell-1}.
\end{equation}
Moreover,
\begin{equation}\label{eq:sector-gap}
 \gamma_\ell^\perp-\gamma_\ell
 =\ell2^{\ell-1}
 \left(1+\frac{\ell}{2^\ell-1}\right)>0.
\end{equation}
The spectral equation has a unique root in each of the three first-shell
components.  The invariant root is simple, while the roots associated with
$\cE_N^{\mathrm{ev},0}$ and $\cE_N^{\mathrm{odd}}$ have multiplicities
$N-1$ and $N$, respectively.  These roots satisfy
\begin{align}
 z_{\mathrm{triv},N}
 &=\Lambda_N+\gamma_\ell N^{\ell-1}
 +O_\ell(N^{\ell-3/2}),\label{eq:roottriv}\\
 z_{\mathrm{ev},N},\ z_{\mathrm{odd},N}
 &=\Lambda_N+\gamma_\ell^\perp N^{\ell-1}
 +O_\ell(N^{\ell-3/2}).\label{eq:rootnontriv}
\end{align}
Thus the bottom of the first-shell cluster lies in the hyperoctahedrally
invariant sector and is simple for all sufficiently large $N$.
\end{proposition}

\begin{proof}
Let
\[
 \mathcal H_{\mathrm{eff}}(z)=P_1K_{\ell,N}P_1
 -P_1K_{\ell,N}Q_1
 \bigl(Q_1(K_{\ell,N}-z)Q_1\bigr)^{-1}
 Q_1K_{\ell,N}P_1
\]
be the energy-dependent effective Hamiltonian.  Lemmas~\ref{lem:schur}
and~\ref{lem:resolvent} make it well defined uniformly in the stated window
and identify the equation $K_{\ell,N}u=zu$ with
$\mathcal H_{\mathrm{eff}}(z)P_1u=zP_1u$.
The projections $P_1,Q_1$ commute with $\BN$, as do $K_{\ell,N}$ and the
compressed resolvent.  Hence $\mathcal H_{\mathrm{eff}}(z)$ is
$\BN$-equivariant.  This justifies its scalar restriction to each irreducible
summand of \eqref{eq:decomposition} below.

We first localize the bottom without using a spectral parameter.  The $q=2$
form bound for $Q_1K_{\ell,N}Q_1$ and \eqref{eq:blocksPKQ}, together with
completion of the square in the block decomposition $P_1+Q_1$, give
\begin{equation}\label{eq:coarselower}
 K_{\ell,N}\geq
 \bigl(\Lambda_N-A_\ell N^{\ell-1}\bigr)I
\end{equation}
for a constant $A_\ell>0$ independent of $N$.  Indeed, if $p=P_1v$ and
$q=Q_1v$, then
\[
 2\operatorname{Re}\ip{P_1K_{\ell,N}Q_1q}{p}
 \geq-\frac{c_\ell N^\ell}{2}\norm q^2
      -A_\ell N^{\ell-1}\norm p^2.
\]
This is the inequality
$2|\ip{x}{y}|\leq\varepsilon\norm{x}^2+\varepsilon^{-1}\norm{y}^2$
with $\varepsilon$ chosen at the scale $N^\ell$.  It follows from
\eqref{eq:blocksPKQ}, because the squared $P_1$--$Q_1$ coupling is
$O(N^{2\ell-1})$ and is divided by a $Q_1$ form bound of size $N^\ell$.
Combining it with
$P_1K_{\ell,N}P_1=\Lambda_NP_1+O_\ell(N^{\ell-1})$ proves
\eqref{eq:coarselower}.  The first-shell trial space gives the matching
upper localization $\inf\sigma(K_{\ell,N})\leq
\Lambda_N+A_\ell N^{\ell-1}$.  Thus the stated Feshbach window contains all
spectral points relevant to the bottom.

Insert \eqref{eq:blocksPKP}, \eqref{eq:blocksPKQ}, and
\eqref{eq:resolvent11} into the Schur term.  Its
leading part is
\[
 \frac{\ell^2(2N)^{2\ell-2}2^\ell}
 {(2^\ell-1)\Lambda_N}\,\mathcal B_N
 =\frac{\ell^22^\ell}{2^\ell-1}
 (2N)^{\ell-2}\mathcal B_N.
\]
The cross terms containing $T_{\ell,N}$ are
$O_\ell(N^{\ell-3/2})$: one leading coupling has size
$O_\ell(N^{\ell-1/2})$, the resolvent has size $O_\ell(N^{-\ell})$, and the
remainder coupling has size $O_\ell(N^{\ell-1})$.  The square of the remainder
is $O_\ell(N^{\ell-2})$, and the $O_\ell(N^{-\ell-1})$ error in
\eqref{eq:resolvent11}, multiplied by the square of the leading coupling,
also contributes $O_\ell(N^{\ell-2})$.  Combining this with
\eqref{eq:blocksPKP} proves \eqref{eq:feshbach} with the stated operator-norm
remainder.

Finally, substitute the exact eigenvalues of $\mathcal J_N$ and
$\mathcal B_N$.  On $\cE_N^{\mathrm{triv}}$ their leading values are $2N$
and $4N$, respectively.  Hence the coefficient of $N^{\ell-1}$ in
\eqref{eq:feshbach} is
\[
 2^{\ell-2}\left(
 6\binom\ell2-\frac{4\ell^22^\ell}{2^\ell-1}\right)
 =\gamma_\ell.
\]
On both non-trivial components $\mathcal J_N$ vanishes and the leading value
of $\mathcal B_N$ is $2N$.  Their common coefficient is therefore
\[
 2^{\ell-2}\left(
 2\binom\ell2-\frac{2\ell^22^\ell}{2^\ell-1}\right)
 =\gamma_\ell^\perp.
\]
The constant parts $-4$ and $-2$ of the two exact $\mathcal B_N$ eigenvalues
only contribute at order $N^{\ell-2}$ and do not alter this comparison.
Subtracting the two displayed coefficients and simplifying gives
\eqref{eq:sector-gap}.  Its right-hand side is strictly positive.

It remains to incorporate the spectral parameter.  Since the three
representations in \eqref{eq:decomposition} are irreducible and occur with
multiplicity one, the restriction of $\mathcal H_{\mathrm{eff}}(z)$ to each
one is a scalar $h_{\tau,N}(z)$, where
$\tau\in\{\mathrm{triv},\mathrm{ev},\mathrm{odd}\}$.  The expansion just proved
is uniform in the energy window and gives
\begin{equation}\label{eq:heff-scalar}
 h_{\tau,N}(z)=\Lambda_N+a_{\tau,\ell}N^{\ell-1}
 +O_\ell(N^{\ell-3/2}),
 \qquad
 a_{\mathrm{triv},\ell}=\gamma_\ell,
 \quad
 a_{\mathrm{ev},\ell}=a_{\mathrm{odd},\ell}=\gamma_\ell^\perp.
\end{equation}
If $D(z)=Q_1(K_{\ell,N}-z)Q_1$, differentiation inside the gap gives
\[
 \frac{d}{dz}D(z)^{-1}=D(z)^{-2},\qquad
 h_{\tau,N}'(z)
 =-\norm{D(z)^{-1}Q_1K_{\ell,N}p_\tau}^2\leq0,
\]
where $p_\tau$ is any unit vector in the corresponding first-shell component.
Thus $h_{\tau,N}(z)-z$ is strictly decreasing.  Choose a fixed $M$ larger
than the uniform constant in the remainder in \eqref{eq:heff-scalar} and set
\[
 z_\pm=\Lambda_N+a_{\tau,\ell}N^{\ell-1}
 \pm MN^{\ell-3/2}.
\]
Then $h_{\tau,N}(z_-)-z_->0$ and
$h_{\tau,N}(z_+)-z_+<0$.  Continuity and strict monotonicity give a unique
root between $z_-$ and $z_+$.  Lemma~\ref{lem:schur} identifies it with the
unique spectral value attached to that first-shell representation.  The kernel
correspondence in Lemma~\ref{lem:schur} shows that its multiplicity equals the
dimension of the corresponding irreducible first-shell space, namely $1$,
$N-1$, or $N$.  This proves
\eqref{eq:roottriv}--\eqref{eq:rootnontriv}.  Finally,
\eqref{eq:sector-gap} and the $O(N^{\ell-3/2})$ root errors show that the invariant root is the unique
lowest component for all sufficiently large $N$; its multiplicity is one.
\end{proof}

\section{The Fourier gap and confinement}

The Fourier representation used by Gupta and Huang--Ye is particularly well suited
to the complement of the first shell.  Throughout this section
$d\mu=(2\pi)^{-N}dx$, and all torus integrals are taken with respect to $d\mu$.
We use the non-positive torus Laplacian
$\Delta=\sum_{j=1}^N\partial_{x_j}^2$; thus
$\Delta e^{in\cdot x}=-|n|^2e^{in\cdot x}$.  Let
\[
 \omega(x)=\sum_{j=1}^N\sin^2(x_j/2),\qquad x\in\T^N,
\]
and use the convention
\[
 D^k\phi=
 \begin{cases}
  \Delta^{k/2}\phi,&k\ \text{even},\\
  \nabla\Delta^{(k-1)/2}\phi,&k\ \text{odd}.
 \end{cases}
\]
Our Fourier convention is
\[
 \widehat f(n)=\int_{\T^N}f(x)e^{-in\cdot x}\,d\mu(x),
 \qquad
 f(x)=\sum_{n\in\Z^N}\widehat f(n)e^{in\cdot x}.
\]
For $u\in C_c(\Z^N)$ with $u(0)=0$, define the trigonometric polynomial
$\Psi$ by
\[
 \widehat\Psi(0)=0,
 \qquad
 \widehat\Psi(n)=\frac{u(n)}{|n|^{2\ell}}\quad(n\ne0).
\]
It is smooth and has zero mean.  Parseval's identity and the definition of
$D^\ell$ give
\begin{align}
 \sum_{n\ne0}\frac{|u(n)|^2}{|n|^{2\ell}}
 &=\int_{\T^N}|D^\ell\Psi|^2,\label{eq:fourierden}\\
 \sum_n|D^\ell u(n)|^2
 &=4^\ell\int_{\T^N}|D^{2\ell}\Psi|^2\omega^\ell.
 \label{eq:fouriernum}
\end{align}
Indeed, the Fourier coefficient of $D^{2\ell}\Psi=\Delta^\ell\Psi$ is
$(-1)^\ell u(n)$.  The discrete Fourier symbol of $L_N$ is
\[
 2N-2\sum_{j=1}^N\cos x_j=4\omega(x),
\]
and hence Parseval yields \eqref{eq:fouriernum}; the sign $(-1)^\ell$ is
immaterial.  Thus $u$ vanishes on $S_1$ precisely when $\Psi$ has no Fourier
modes with $|n|^2=1$.  These identities hold first on the finitely supported
form core.  Density and closure extend them, and the inequalities derived
from them, to the compressed form domains used below.

For $q\in\N$, let
\[
 \mathscr F_{q,N}
 =\{\phi:\widehat\phi(n)=0\text{ whenever }|n|^2<q\}.
\]

Write $\rho=2\omega/N$, and let $P_{<q}$ be the
orthogonal Fourier projection onto $|n|^2<q$.

The following multiplier estimate shows that multiplication by a fixed power
of $\rho$ creates only a small low-frequency component.

\begin{lemma}\label{lem:multiplier}
For every fixed $q\in\N$ and $\beta>0$, there are
$N^{\mathrm{mult}}_{\beta,q}\in\N$ and positive constants
$C_{\beta,q},c_\beta$ such that, for all
$N\geq N^{\mathrm{mult}}_{\beta,q}$,
\begin{equation}\label{eq:multiplier}
 \sup_{\substack{g\in\operatorname{ran}P_{<q}\\ \norm g_2=1}}
 \int_{\T^N}|1-\rho^{-\beta}|^2|g|^2\,d\mu
 \leq \frac{C_{\beta,q}}N+C_{\beta,q}e^{-c_{\beta}N}.
\end{equation}
Consequently, for the same dimensions, if $h=\rho^\beta$,
$\phi\in\mathscr F_{q,N}$ and $v=h\phi$,
then
\begin{equation}\label{eq:lowprojection}
 \norm{P_{<q}v}_2\leq C_{\beta,q}N^{-1/2}\norm v_2.
\end{equation}
\end{lemma}

\begin{proof}
Let $H_t=e^{t\Delta}$ be the heat semigroup on the product torus.  The
dimension-free logarithmic Sobolev inequality of Gross \cite{Gross1975}
provides a time $t_0>0$, independent of $N$, such that
$\norm{H_{t_0}f}_4\leq\norm f_2$.  If
$g\in\operatorname{ran}P_{<q}$, put $f=H_{-t_0}g$.  Since
$|n|^2<q$ on the Fourier support of $g$,
\begin{equation}\label{eq:hypercontractive}
 \norm g_4\leq\norm f_2\leq e^{t_0q}\norm g_2=:C_q\norm g_2,
 \qquad g\in\operatorname{ran}P_{<q},
\end{equation}
with $C_q$ independent of $N$.

On $\{\rho\geq1/2\}$, the map $t\mapsto t^{-\beta}$ is Lipschitz, so
$|1-\rho^{-\beta}|\leq C_\beta|1-\rho|$.  The fourth
central moment of the empirical mean
$\rho=N^{-1}\sum_{j=1}^N(1-\cos x_j)$ is $O(N^{-2})$.  Indeed, if
$Z_j=-\cos x_j$, then the $Z_j$ are independent, centered,
$\mathbb E Z_j^2=1/2$, and $\mathbb E Z_j^4=3/8$, so
\[
 \mathbb E(\rho-1)^4
 =\frac{N(3/8)+3N(N-1)(1/2)^2}{N^4}
 \leq\frac{C}{N^2}.
\]
Hence H\"older's inequality and \eqref{eq:hypercontractive} give
\[
 \int_{\{\rho\geq1/2\}}|1-\rho^{-\beta}|^2|g|^2\,d\mu
 \leq C_\beta
 \left(\int|\rho-1|^4d\mu\right)^{1/2}\norm g_4^2
 \leq C_{\beta,q}N^{-1}\norm g_2^2.
\]
For the remaining region we first derive the required small-ball estimate.
If $X$ is uniform on $[-\pi,\pi]$ and $Y=1-\cos X$, then
$Y\geq2X^2/\pi^2$.  Hence, for $s\geq1$,
\[
 \mathbb E e^{-sY}
 \leq\frac1{2\pi}\int_{-\pi}^{\pi}e^{-2sx^2/\pi^2}\,dx
 \leq C s^{-1/2}.
\]
Independence and exponential Markov inequality with $s=r^{-1}$ now yield
\[
 \begin{split}
 \mu\{\rho\leq r\}
 &=\mathbb P\left\{e^{-r^{-1}\sum_{j=1}^NY_j}\geq e^{-N}\right\}\\
 &\leq e^{N}\bigl(\mathbb E e^{-Y/r}\bigr)^N
 \leq(C_0r)^{N/2},\qquad 0<r<1/2.
 \end{split}
\]
Hoeffding's inequality \cite{Hoeffding1963} also gives
$\mu\{\rho<1/2\}\leq e^{-cN}$.  Choose
$r_0<\min(1/2,(2C_0)^{-1})$.  On $r_0\leq\rho<1/2$ the inverse weight is
bounded and Hoeffding applies.  On the intervals
$2^{-k-1}r_0\leq\rho<2^{-k}r_0$,
\begin{align*}
 \int_{\{\rho<r_0\}}\rho^{-4\beta}\,d\mu
 &\leq r_0^{-4\beta}
 \sum_{k\geq0}2^{4\beta(k+1)}
 \mu\{\rho<2^{-k}r_0\}\\
 &\leq C_\beta
 \sum_{k\geq0}2^{4\beta k}2^{-(k+1)N/2}
 \leq C_\beta e^{-c_\beta N}
\end{align*}
for $N>8\beta+2$.  Therefore
\[
 \int_{\{\rho<1/2\}}\rho^{-4\beta}\,d\mu
 \leq C_\beta e^{-c_\beta N}.
\]
The restriction to large dimensions is essential.  Already for the constant
function $g=1$, one has
$\rho^{-2\beta}\asymp |x|^{-4\beta}$ near the origin, so the integral in
\eqref{eq:multiplier} is locally finite only when $N>4\beta$.  The present
fourth-moment argument uses the stronger sufficient condition
$N>8\beta+2$, which is absorbed into
$N^{\mathrm{mult}}_{\beta,q}$.
On $\{\rho<1/2\}$ we have
$|1-\rho^{-\beta}|^4\leq C_\beta(1+\rho^{-4\beta})$.  A second application of
H\"older's inequality and \eqref{eq:hypercontractive} therefore gives
\[
 \int_{\{\rho<1/2\}}|1-\rho^{-\beta}|^2|g|^2d\mu
 \leq C_{\beta,q}e^{-c_\beta N}\norm g_2^2.
\]
Together with the estimate on $\{\rho\geq1/2\}$, this proves
\eqref{eq:multiplier}.

Finally $P_{<q}\phi=0$.  For every $g\in\operatorname{ran}P_{<q}$,
\[
 0=\ip{\phi}{g}=\ip{h^{-1}v}{g},\qquad
 \ip{v}{g}=\ip{v}{(1-h^{-1})g}.
\]
Taking absolute values, applying the standard $L^2$ estimate
\[
 |\ip{v}{(1-h^{-1})g}|
 \leq\norm v_2\norm{(1-h^{-1})g}_2,
\]
and then using \eqref{eq:multiplier}, we obtain
\[
 |\ip{v}{g}|\leq C_{\beta,q}N^{-1/2}\norm v_2\norm g_2.
\]
Taking the supremum over unit vectors $g$ proves
\eqref{eq:lowprojection}.  For complex functions, apply the same argument
to the Hermitian inner product.  The zero of $h$ at the origin is handled by
$h_\varepsilon=(\rho+\varepsilon)^\beta$; the inverse-moment estimate above
provides an integrable majorant, so dominated convergence applies as
$\varepsilon\downarrow0$.
\end{proof}

The following approximation statement permits the use of power
ground states without excluding a neighbourhood of the origin.

\begin{lemma}\label{lem:regularization}
Let $\gamma\geq1$, let $\alpha<0$, and let $\phi$ be smooth on $\T^N$.  Put
$V=\omega^\gamma$ and
\[
 f_\varepsilon=(\omega+\varepsilon)^\alpha,
 \qquad h_\varepsilon=(\rho+\varepsilon)^{\gamma/2}.
\]
The first-order ground-state identity and the second-order periodic Rellich
identity are valid with $f_\varepsilon$, and the conjugation identity for
$\int h_\varepsilon^2|\nabla\phi|^2$ is valid with $h_\varepsilon$.  As
$\varepsilon\downarrow0$, all coefficient terms converge in $L^1$, while the
non-negative square terms are lower semicontinuous.  Consequently the
inequalities obtained by discarding those squares remain valid for
$f=\omega^\alpha$ and $h=\rho^{\gamma/2}$.
\end{lemma}

\begin{proof}
For every $\varepsilon>0$, the functions $f_\varepsilon$ and
$h_\varepsilon$ are smooth and strictly positive, while $V$ is smooth and
non-negative.  To record the signs in the second-order formula, two
integrations by parts first give the weighted Bochner identity
\[
 \int V|\Delta\phi|^2
 =\int V|\nabla^2\phi|^2-
 \int(\Delta V)|\nabla\phi|^2
 +\operatorname{Re}\sum_{i,j}\int V_{ij}\phi_i\overline{\phi_j}.
\]
Applying the first-order ground-state identity to each derivative $\phi_j$
converts the Hessian term into the last square displayed below.  Thus
integration by parts gives, with the complex conjugates shown explicitly,
\begin{align*}
 \int V|\nabla\phi|^2
 &=-\int\frac{\operatorname{div}(V\nabla f_\varepsilon)}
 {f_\varepsilon}|\phi|^2
 +\int V\left|\nabla\phi-
 \frac{\nabla f_\varepsilon}{f_\varepsilon}\phi\right|^2,\\
 \int V|\Delta\phi|^2
 &=-\int\left(
 \frac{\operatorname{div}(V\nabla f_\varepsilon)}{f_\varepsilon}
 +\Delta V\right)|\nabla\phi|^2
 +\operatorname{Re}\sum_{i,j}\int V_{ij}\phi_i\overline{\phi_j}\\
 &\quad+
 \int V\left|\nabla^2\phi-
 \frac{\nabla f_\varepsilon}{f_\varepsilon}
 \otimes\nabla\phi\right|^2,
\end{align*}
and
\[
 \int h_\varepsilon^2|\nabla\phi|^2
 =\int|\nabla(h_\varepsilon\phi)|^2
 +\int\frac{\Delta h_\varepsilon}{h_\varepsilon}
 |h_\varepsilon\phi|^2.
\]
We verify the limiting coefficients explicitly.  Differentiation gives
\begin{align}
 \frac{\operatorname{div}(V\nabla f_\varepsilon)}
 {f_\varepsilon}
={}&\alpha\frac{\omega^\gamma}{\omega+\varepsilon}\Delta\omega
 +\alpha\gamma\frac{\omega^{\gamma-1}}{\omega+\varepsilon}
 |\nabla\omega|^2 \nonumber\\
 &+\alpha(\alpha-1)
 \frac{\omega^\gamma}{(\omega+\varepsilon)^2}|\nabla\omega|^2.
 \label{eq:regularized-f-coefficient}
\end{align}
Since
\[
 |\nabla\omega|^2\leq\omega,
 \qquad |\Delta\omega|\leq N/2+\omega,
 \qquad 0\leq\frac{\omega}{\omega+\varepsilon}\leq1,
\]
the absolute value of \eqref{eq:regularized-f-coefficient} is bounded, for
fixed $N,\alpha,\gamma$, by
$C_{N,\alpha,\gamma}(\omega^{\gamma-1}+\omega^\gamma)$.  It converges almost
everywhere to
\begin{equation}\label{eq:limit-f-coefficient}
 \alpha(\alpha+\gamma-1)\omega^{\gamma-2}|\nabla\omega|^2
 +\alpha\omega^{\gamma-1}\Delta\omega.
\end{equation}
The derivatives of the weight are
\[
 V_{ij}=\gamma(\gamma-1)\omega^{\gamma-2}\omega_i\omega_j
 +\gamma\omega^{\gamma-1}\omega_{ij}.
\]
The quadratic form of the first term is
$\gamma(\gamma-1)\omega^{\gamma-2}
|\sum_i\omega_i\phi_i|^2$, which is bounded by
$C_\gamma\omega^{\gamma-1}|\nabla\phi|^2$ because
$|\nabla\omega|^2\leq\omega$.  The second term has
$\omega_{ij}=\frac12\cos x_i\,\delta_{ij}$ and obeys the same bound.  Thus all coefficients in the two
$f_\varepsilon$ identities have an integrable majorant independent of
$\varepsilon$.

For the third identity, with $\beta=\gamma/2$,
\begin{align}
 h_\varepsilon\Delta h_\varepsilon
={}&\beta(\rho+\varepsilon)^{\gamma-1}\Delta\rho
 +\beta(\beta-1)(\rho+\varepsilon)^{\gamma-2}
 |\nabla\rho|^2.\label{eq:regularized-h-coefficient}
\end{align}
The inequality
$|\nabla\rho|^2\leq(2\rho-\rho^2)/N\leq2(\rho+\varepsilon)/N$ shows that the
second term is bounded by a constant times
$(\rho+\varepsilon)^{\gamma-1}$.  Since $0\leq\rho\leq2$ and $\gamma\geq1$,
\eqref{eq:regularized-h-coefficient} has an integrable bound independent of
$\varepsilon$ and converges almost everywhere to $h\Delta h$.

Dominated convergence now applies to every coefficient that is retained.
Fatou's lemma applies to each non-negative square, so discarding that square
before taking the limit preserves the resulting inequality.  The order of
operations used below is therefore unambiguous: first fix $N$ and the value
of $\alpha$ (eventually $\alpha=-\sqrt N$), next use the regularized
identities, then let $\varepsilon\downarrow0$, and only afterwards estimate
the resulting explicit coefficients as $N\to\infty$.  No dominated-convergence
claim uniform in $\alpha$ or $N$ is made or needed.  This proves the lemma.
\end{proof}

We next combine the multiplier lemma with weighted ground-state identities.

\begin{proposition}\label{prop:weightedgap}
Fix $q,\gamma\in\N$.  Uniformly for smooth
$\phi\in\mathscr F_{q,N}$,
\begin{align}
 \int|\nabla\phi|^2\omega^\gamma
 &\geq(q-o(1))\int|\phi|^2\omega^\gamma,
 \label{eq:weightedPoincare}\\
 \int|\nabla\phi|^2\omega^\gamma
 &\geq\frac{qN}{2}(1-o(1))
 \int|\phi|^2\omega^{\gamma-1},
 \label{eq:weightedHardy}\\
 \int|\Delta\phi|^2\omega^\gamma
 &\geq(q-o(1))\int|\nabla\phi|^2\omega^\gamma,
 \label{eq:weightedRellichPoincare}\\
 \int|\Delta\phi|^2\omega^\gamma
 &\geq\frac{qN}{2}(1-o(1))
 \int|\nabla\phi|^2\omega^{\gamma-1}.
 \label{eq:weightedRellich}
\end{align}
More precisely, there are constants $A_{q,\gamma}>0$ and
$N_{q,\gamma}$ such that, for $N\geq N_{q,\gamma}$, every occurrence of
$o(1)$ in these four inequalities may be replaced in absolute value by
$A_{q,\gamma}N^{-1/2}$.
\end{proposition}

\begin{proof}
Set $\beta=\gamma/2$, $h=\rho^\beta$, and $v=h\phi$.  We first use
$h_\varepsilon$ from Lemma~\ref{lem:regularization} and then let
$\varepsilon\downarrow0$.  The resulting ground-state identity is
\[
 \int\rho^\gamma|\nabla\phi|^2d\mu
 =\int|\nabla v|^2d\mu+\int\frac{\Delta h}{h}|v|^2d\mu.
\]
Since
\[
 |\nabla\rho|^2\leq\frac{2\rho-\rho^2}{N},
 \qquad \Delta\rho=1-\rho,
\]
we have
\[
 \frac{\Delta h}{h}
 =\beta\frac{1-\rho}{\rho}
 +\beta(\beta-1)\frac{|\nabla\rho|^2}{\rho^2}.
\]
We now bound the negative part explicitly.  If $\rho\geq1$, the first term has
negative part at most $\beta(\rho-1)$; the second term is non-negative when
$\beta\geq1$, while for $0<\beta<1$ its absolute value is at most
$2\beta(1-\beta)/(N\rho)\leq C_\gamma/N$.  If $0<\rho<1$ and
$\beta\geq1$, both terms are non-negative.  If $0<\beta<1$, then
\[
 \frac{\Delta h}{h}
 \geq\frac{\beta}{\rho}
 \left[(1-\rho)-\frac{(1-\beta)(2-\rho)}N\right].
\]
The right-hand side can be negative only if $1-\rho\leq2/N$.  For $N\geq4$
this implies $\rho\geq1/2$, and the magnitude of the negative part is at most
$C_\gamma/N$.  We have therefore proved
\begin{equation}\label{eq:Wpointwise}
 W=(-\Delta h/h)_+
 \leq C_\gamma\bigl((\rho-1)_++N^{-1}\bigr).
\end{equation}
For a non-negative integrable function $F$, we use the normalized entropy
\[
 \operatorname{Ent}_\mu(F)
 =\int F\log\!\left(\frac{F}{\int F\,d\mu}\right)d\mu.
\]
Apply the entropy variational inequality
\cite[Chapter~4]{BoucheronLugosiMassart} with parameter $\sqrt N$,
Hoeffding's exponential-moment bound \cite{Hoeffding1963} for
$\sqrt N(\rho-1)$, and the dimension-free logarithmic Sobolev inequality.
More explicitly, for $s=\sqrt N$,
\[
 \int W|v|^2d\mu
 \leq\frac1s\operatorname{Ent}_\mu(|v|^2)
 +\frac1s\log\left(\int e^{sW}d\mu\right)\norm v_2^2.
\]
Equation \eqref{eq:Wpointwise} and
$e^{a x_+}\leq1+e^{ax}$ imply
\[
 \int e^{sW}d\mu
 \leq e^{C_\gamma s/N}
 \left(1+\int e^{C_\gamma s(\rho-1)}d\mu\right).
\]
Since $\rho-1=N^{-1}\sum_{j=1}^N(-\cos x_j)$ is the average of independent,
centered random variables in $[-1,1]$, Hoeffding's exponential-moment estimate
with $s=\sqrt N$ bounds the last integral by
$\exp(C_\gamma^2s^2/(2N))$.  Thus the logarithm is $O_\gamma(1)$, uniformly in
$N$.  Gross's inequality bounds the entropy by
$C_{\mathrm{LS}}\int|\nabla v|^2d\mu$.  Consequently
\begin{equation}\label{eq:Wform}
 \int W|v|^2d\mu
 \leq C_\gamma N^{-1/2}
 \left(\int|\nabla v|^2d\mu+\int|v|^2d\mu\right).
\end{equation}
Lemma~\ref{lem:multiplier} and the unweighted spectral theorem give
\[
 \int|\nabla v|^2d\mu
 \geq q\bigl(\norm v_2^2-\norm{P_{<q}v}_2^2\bigr)
 \geq(q-C_{q,\gamma}N^{-1})\norm v_2^2.
\]
Combining this with \eqref{eq:Wform} proves
\eqref{eq:weightedPoincare} with an $O(N^{-1/2})$ error.  More explicitly,
the ground-state identity and \eqref{eq:Wform} imply
\[
 \int\rho^\gamma|\nabla\phi|^2d\mu
 \geq(1-C_\gamma N^{-1/2})\norm{\nabla v}_2^2
 -C_\gamma N^{-1/2}\norm v_2^2.
\]
The preceding spectral bound for $v$ shows that there is a constant
$B_{q,\gamma}>0$, independent of $N$, for which the right-hand side is at
least
$(q-B_{q,\gamma}N^{-1/2})\norm v_2^2$.  Choose
$A_{q,\gamma}^{(1)}\geq B_{q,\gamma}$ and set
\[
 \kappa_{q,\gamma,N}=q-A_{q,\gamma}^{(1)}N^{-1/2}.
\]
For $N\geq N_{q,\gamma}^{(1)}$, this number is positive and the estimate just
proved reads
\[
 \int|\nabla\phi|^2\rho^\gamma d\mu
 \geq\kappa_{q,\gamma,N}
 \int|\phi|^2\rho^\gamma d\mu.
\]
Since
$\norm v_2^2=\int\rho^\gamma|\phi|^2d\mu$, multiplying both sides by
$(N/2)^\gamma$ gives \eqref{eq:weightedPoincare} with the weight
$\omega^\gamma$.

For the Hardy estimate, apply Lemma~\ref{lem:regularization} with
$V=\omega^\gamma$ and $f=\omega^\alpha$.  The relations
$|\nabla\omega|^2\leq\omega$ and $\Delta\omega=N/2-\omega$ imply, for
$\alpha<0$ with $\alpha(\alpha+\gamma-1)\geq0$,
\begin{equation}\label{eq:Hardyalpha}
 I_\gamma\geq
 -\alpha(\alpha+\gamma-1+N/2)V_{\gamma-1}+\alpha V_\gamma,
\end{equation}
where
\[
 I_\eta=\int|\nabla\phi|^2\omega^\eta d\mu,
 \qquad V_\eta=\int|\phi|^2\omega^\eta d\mu.
\]
Indeed,
\[
 \frac{\operatorname{div}(\omega^\gamma\nabla\omega^\alpha)}
 {\omega^\alpha}
 =\alpha(\alpha+\gamma-1)\omega^{\gamma-2}|\nabla\omega|^2
 +\alpha\omega^{\gamma-1}\Delta\omega.
\]
The ground-state identity has the negative of this expression as its
coefficient.  The sign condition on
$\alpha(\alpha+\gamma-1)$ permits the use of
$|\nabla\omega|^2\leq\omega$, and substitution of
$\Delta\omega=N/2-\omega$ gives \eqref{eq:Hardyalpha}.

By the quantitative form of \eqref{eq:weightedPoincare},
\[
 V_\gamma\leq\kappa_{q,\gamma,N}^{-1}I_\gamma.
\]
Because $\alpha<0$, insertion
in \eqref{eq:Hardyalpha} reverses the inequality when
$V_\gamma\leq\kappa_{q,\gamma,N}^{-1}I_\gamma$ is multiplied by $\alpha$.
Consequently,
\[
 \left(1-\frac{\alpha}{\kappa_{q,\gamma,N}}\right)I_\gamma
 \geq-\alpha(\alpha+\gamma-1+N/2)V_{\gamma-1}.
\]
For $N$ large, the choice $\alpha=-\sqrt N$ satisfies the required sign
condition and gives
\[
 \frac{-\alpha(\alpha+\gamma-1+N/2)}
 {1-\alpha/\kappa_{q,\gamma,N}}
 =\frac{qN}{2}(1-O_{q,\gamma}(N^{-1/2})),
\]
which proves \eqref{eq:weightedHardy}.

For the second-order spectral estimate, integration by parts and Young's
inequality with a free parameter $a>0$ give
\[
 I_\gamma\leq \frac12\int\Delta(\omega^\gamma)|\phi|^2d\mu
 +\frac a2V_\gamma+\frac1{2a}J_\gamma,
\]
where $J_\gamma=\int|\Delta\phi|^2\omega^\gamma d\mu$.  To estimate the
first term, note that the preceding inequality follows from
\[
 I_\gamma
 =-\operatorname{Re}\int\omega^\gamma\overline\phi\,\Delta\phi\,d\mu
 +\frac12\int\Delta(\omega^\gamma)|\phi|^2d\mu
\]
and
$2|\phi\Delta\phi|\leq a|\phi|^2+a^{-1}|\Delta\phi|^2$.
To estimate the
weight term, integrate by parts:
\begin{align*}
 \int\omega^{\gamma-1}\Delta\omega|\phi|^2d\mu
 &=-\int\nabla(\omega^{\gamma-1}|\phi|^2)\cdot\nabla\omega\,d\mu\\
 &\leq C_\gamma V_{\gamma-1}
 +2V_{\gamma-1}^{1/2}I_\gamma^{1/2}.
\end{align*}
Equation \eqref{eq:weightedHardy} gives
$V_{\gamma-1}\leq C_{q,\gamma}N^{-1}I_\gamma$.  Since
$|\nabla\omega|^2\leq\omega$, the remaining term in
$\Delta(\omega^\gamma)$ obeys the same bound, and therefore
\begin{equation}\label{eq:Deltawbound}
 \frac12\int\Delta(\omega^\gamma)|\phi|^2d\mu
 \leq C_{q,\gamma}N^{-1/2}I_\gamma.
\end{equation}
Use $V_\gamma\leq\kappa_{q,\gamma,N}^{-1}I_\gamma$ and choose
$a=\kappa_{q,\gamma,N}$.
The two Young terms then combine to give, for a fixed
$A_{q,\gamma}^{(2)}>0$ and all sufficiently large $N$,
\[
 \bigl(1-A_{q,\gamma}^{(2)}N^{-1/2}\bigr)I_\gamma
 \leq\frac{1}{\kappa_{q,\gamma,N}}J_\gamma.
\]
Define
\[
 \widetilde\kappa_{q,\gamma,N}
 =\kappa_{q,\gamma,N}
  \bigl(1-A_{q,\gamma}^{(2)}N^{-1/2}\bigr).
\]
After enlarging the lower threshold for $N$, this number is positive and
\[
 \widetilde\kappa_{q,\gamma,N}
 \geq q-A_{q,\gamma}^{(3)}N^{-1/2},
 \qquad
 J_\gamma\geq\widetilde\kappa_{q,\gamma,N}I_\gamma
\]
for a constant $A_{q,\gamma}^{(3)}>0$.  This proves
\eqref{eq:weightedRellichPoincare}.

For the final Rellich estimate we use the following periodic identity, obtained
by expanding the square and integrating twice by parts:
\begin{align*}
 \int V|\Delta\phi|^2d\mu
={}&-\int\left(\frac{\operatorname{div}(V\nabla f)}f+\Delta V\right)
 |\nabla\phi|^2d\mu
 +\operatorname{Re}\sum_{i,j}\int V_{ij}\phi_i\overline{\phi_j}\,d\mu\\
 &+\int V\left|\nabla^2\phi-
 \frac{\nabla f}{f}\otimes\nabla\phi\right|^2d\mu.
\end{align*}
Take $V=\omega^\gamma$ and $f=\omega^\alpha$, justified by
Lemma~\ref{lem:regularization}.  The last integral is non-negative.  Moreover,
with $\omega_i=\frac12\sin x_i$ and
$\omega_{ij}=\frac12\cos x_i\,\delta_{ij}$,
\begin{align*}
 \operatorname{Re}\sum_{i,j}\int V_{ij}\phi_i\overline{\phi_j}\,d\mu
={}&\gamma(\gamma-1)\int\omega^{\gamma-2}
 \left|\sum_i\omega_i\phi_i\right|^2d\mu\\
 &+\frac\gamma2\sum_i\int\omega^{\gamma-1}
 \cos x_i\,|\phi_i|^2d\mu\\
\geq{}&-\frac\gamma2 I_{\gamma-1}.
\end{align*}
If $\alpha+\gamma<0$, then
$\alpha(\alpha+\gamma-1)+\gamma(\gamma-1)\geq0$.  Direct differentiation,
$\Delta\omega=N/2-\omega$, gives the exact coefficient
\begin{align*}
 \frac{\operatorname{div}(V\nabla f)}f+\Delta V
={}&\bigl[\alpha(\alpha+\gamma-1)+\gamma(\gamma-1)\bigr]
 \omega^{\gamma-2}|\nabla\omega|^2\\
 &+(\alpha+\gamma)\omega^{\gamma-1}(N/2-\omega).
\end{align*}
The coefficient in square brackets is non-negative under the stated
condition.  Using $|\nabla\omega|^2\leq\omega$ therefore gives
\begin{align*}
 -\left(\frac{\operatorname{div}(V\nabla f)}f+\Delta V\right)
 \geq{}&-\bigl[\alpha^2
 +(\alpha+\gamma)(\gamma-1+N/2)\bigr]\omega^{\gamma-1}\\
 &+(\alpha+\gamma)\omega^\gamma.
\end{align*}
Consequently,
\begin{equation}\label{eq:Rellichalpha}
 J_\gamma\geq-G_{N,\gamma}(\alpha)I_{\gamma-1}
 +(\alpha+\gamma)I_\gamma,
\end{equation}
where
\begin{equation}\label{eq:Gpoly}
 G_{N,\gamma}(\alpha)
 =\alpha^2+(\alpha+\gamma)(\gamma-1+N/2)
 +\frac\gamma2.
\end{equation}
Because $\alpha+\gamma<0$, the quantitative form of
\eqref{eq:weightedRellichPoincare} gives
$ (\alpha+\gamma)I_\gamma
 \geq(\alpha+\gamma)\widetilde\kappa_{q,\gamma,N}^{-1}J_\gamma$.
For $N$ sufficiently large, $\alpha=-\sqrt N$ satisfies
$\alpha+\gamma<0$ and
$\alpha(\alpha+\gamma-1)+\gamma(\gamma-1)\geq0$.  Substituting the preceding bound in
\eqref{eq:Rellichalpha} and moving the resulting multiple of $J_\gamma$ to
the left gives
\[
 \left(1-\frac{\alpha+\gamma}
 {\widetilde\kappa_{q,\gamma,N}}\right)J_\gamma
 \geq-G_{N,\gamma}(\alpha)I_{\gamma-1}.
\]
For $\alpha=-\sqrt N$, the two coefficients satisfy
\begin{align*}
 -G_{N,\gamma}(-\sqrt N)
 &=\tfrac12N^{3/2}+O_\gamma(N),\\
 1-\frac{-\sqrt N+\gamma}{\widetilde\kappa_{q,\gamma,N}}
 &=\frac{\sqrt N}{q}\bigl(1+O_{q,\gamma}(N^{-1/2})\bigr),
\end{align*}
and therefore
\[
 -\frac{G_{N,\gamma}(-\sqrt N)}
 {1-(-\sqrt N+\gamma)/\widetilde\kappa_{q,\gamma,N}}
 =\frac{qN}{2}\bigl(1+O_{q,\gamma}(N^{-1/2})\bigr).
\]
The two quotient expansions used for \eqref{eq:weightedHardy} and
\eqref{eq:weightedRellich} have errors bounded, respectively, by
$A_{q,\gamma}^{(4)}N^{-1/2}$ and
$A_{q,\gamma}^{(5)}N^{-1/2}$ once $N$ exceeds fixed thresholds.  Define
\[
 A_{q,\gamma}
 =\max_{1\leq j\leq5}A_{q,\gamma}^{(j)}
\]
and let $N_{q,\gamma}$ be the maximum of the finitely many lower thresholds
introduced in the proof, including
$N^{\mathrm{mult}}_{\gamma/2,q}$ from Lemma~\ref{lem:multiplier} and those
required by the sign conditions on $\alpha=-\sqrt N$.  Then the errors in all four inequalities
\eqref{eq:weightedPoincare}--\eqref{eq:weightedRellich} are bounded in
absolute value by $A_{q,\gamma}N^{-1/2}$.  This proves the full quantitative
statement of the proposition.
\end{proof}

Iterating the two weighted estimates gives the fixed Fourier gap required by
the operator argument.

\begin{proposition}\label{prop:fouriergap}
For every fixed $\ell,q\in\N$, there are constants
$C_{\ell,q}>0$ and $N_{\ell,q}$ and numbers
\[
 0\leq\varepsilon_{\ell,q}(N)
 \leq C_{\ell,q}N^{-1/2},\qquad N\geq N_{\ell,q},
\]
such that
\begin{equation}\label{eq:fouriergap}
 4^\ell\int_{\T^N}|D^{2\ell}\Psi|^2\omega^\ell
 \geq
 (2q)^\ell N^\ell(1-\varepsilon_{\ell,q}(N))
 \int_{\T^N}|D^\ell\Psi|^2
\end{equation}
for every smooth $\Psi\in\mathscr F_{q,N}$.
\end{proposition}

\begin{proof}
For $\ell\leq r\leq2\ell$, set
\[
 X_r=\int_{\T^N}|D^r\Psi|^2\omega^{r-\ell}d\mu.
\]
Every derivative of $\Psi$ belongs to $\mathscr F_{q,N}$, since
differentiation multiplies Fourier coefficients and cannot create new
frequencies.  We claim that, for $r=\ell+1,\ldots,2\ell$,
\begin{equation}\label{eq:Xstep}
 X_r\geq\frac{qN}{2}
 (1-O_{\ell,q}(N^{-1/2}))X_{r-1}.
\end{equation}

If $r$ is odd, take
$\phi=\Delta^{(r-1)/2}\Psi$ in \eqref{eq:weightedHardy} with
$\gamma=r-\ell$.  Then $|\nabla\phi|^2=|D^r\Psi|^2$ and
$|\phi|^2=|D^{r-1}\Psi|^2$, so the resulting inequality is precisely
\eqref{eq:Xstep}.  If $r$ is even, take
$\phi=\Delta^{(r-2)/2}\Psi$ in \eqref{eq:weightedRellich}, again with
$\gamma=r-\ell$.  In this case
$|\Delta\phi|^2=|D^r\Psi|^2$ and
$|\nabla\phi|^2=|D^{r-1}\Psi|^2$, and the same inequality follows.

Multiplying \eqref{eq:Xstep} for
$r=2\ell,2\ell-1,\ldots,\ell+1$ gives exactly $\ell$ factors.  Because
$\ell$ is fixed, the product of their relative errors is still
$1-O_{\ell,q}(N^{-1/2})$.  Therefore
\[
 \int|D^{2\ell}\Psi|^2\omega^\ell d\mu
 \geq\left(\frac{qN}{2}\right)^\ell
 (1-O_{\ell,q}(N^{-1/2}))
 \int|D^\ell\Psi|^2d\mu.
\]
Multiplication by $4^\ell$ proves \eqref{eq:fouriergap}, with
$\varepsilon_{\ell,q}(N)=C_{\ell,q}N^{-1/2}$ after increasing the constant
and the lower threshold if necessary.
\end{proof}

The case $q=2$ translates the Fourier gap into a coercive bound off the first
lattice shell.

\begin{proposition}\label{prop:confinement}
For the projection $Q_1=I-P_1$ onto functions vanishing on $S_1$,
\begin{equation}\label{eq:confinement}
 Q_1K_{\ell,N}Q_1
 \geq4^\ell N^\ell(1-O_\ell(N^{-1/2}))Q_1.
\end{equation}
In particular, the complement of the first shell is separated from
$(2N)^\ell=2^\ell N^\ell$ by a gap of order $N^\ell$.
\end{proposition}

\begin{proof}
Let $v\in C_c(\Z^N\setminus\{0\})$ satisfy $P_1v=0$ and put
$u=R^\ell v$.  The support condition is unchanged by $R^\ell$, so $u$ also
vanishes at the origin and on $S_1$.  In the Fourier representation
\eqref{eq:fourierden}--\eqref{eq:fouriernum}, the associated function $\Psi$
has zero mean and has no Fourier coefficient with $|n|^2=1$.  Hence
$\Psi\in\mathscr F_{2,N}$.

Proposition~\ref{prop:fouriergap} with $q=2$ gives
\begin{align*}
 \mathfrak k_{\ell,N}[v]
 &=4^\ell\int|D^{2\ell}\Psi|^2\omega^\ell d\mu\\
 &\geq4^\ell N^\ell(1-O_\ell(N^{-1/2}))
 \int|D^\ell\Psi|^2d\mu.
\end{align*}
By \eqref{eq:fourierden}, the last integral equals
$\sum_{n\ne0}|u(n)|^2/|n|^{2\ell}=\norm v_2^2$.  This proves the quadratic
form inequality on the form core.  Closure of the form and density extend it
to the full compressed form domain, proving \eqref{eq:confinement}.  Since
$4^\ell>2^\ell$ for every $\ell\geq1$, its distance from
$\Lambda_N=2^\ell N^\ell$ is a positive multiple of $N^\ell$ for large $N$.
\end{proof}

\section{Residual estimate and closure}

Let $w_{\ell,N}$ be a normalized eigenvector of the four-orbit compression
$M_{\ell,N}$ for $\mu_{\ell,N}$, embedded in the full invariant sector.

The orbit-path distance also controls the defect created by this embedding.

\begin{lemma}\label{lem:residual}
For every fixed $\ell\geq1$,
\begin{equation}\label{eq:residual}
 \norm{(K_{\ell,N}-\mu_{\ell,N})w_{\ell,N}}
 =O_\ell(N^{\ell-3/2}).
\end{equation}
\end{lemma}

\begin{proof}
The analytic perturbation calculation in Proposition~\ref{prop:delta} gives,
after normalizing the $(1)$ component to one, the component bounds
\[
 1,\qquad O_\ell(N^{-1/2}),\qquad
 O_\ell(N^{-1}),\qquad O_\ell(N^{-1})
\]
on $(1)$, $(1,1)$, $(1,1,1)$, and $(2)$, respectively.  Denote the
corresponding decay exponents by
\[
 \tau(1)=0,\qquad \tau(1,1)=\frac12,
 \qquad \tau(1,1,1)=\tau(2)=1.
\]
Let $\mathcal S=\{(1),(1,1),(1,1,1),(2)\}$.
The norm of this four-component vector is $1+O_\ell(N^{-1})$; hence passing
to the normalized eigenvector changes none of these orders.  We may therefore
write
\[
 w_{\ell,N}=\sum_{\lambda\in\mathcal S}a_{\lambda,N}\psi_\lambda,
 \qquad |a_{\lambda,N}|\leq C_\ell N^{-\tau(\lambda)}.
\]

By Lemma~\ref{lem:adjacency}, the boundary of $\mathcal S$ in the non-zero
orbit graph is
\[
 \partial\mathcal S
 =\{(2,1),(1,1,1,1),(3),(2,1,1)\}.
\]
The last orbit is reached from $(1,1,1)$ by increasing one part from $1$ to
$2$.  The empty partition represents the origin and gives no endpoint
contribution because of the factor $R^\ell$; it is therefore excluded from the
following endpoint statement.  Since every dimensional edge has cost $1/2$
and every internal edge has cost $1$, any path from $\mathcal S$ to a non-zero
omitted orbit meets $\partial\mathcal S$ before taking a further edge.  We
claim that
\begin{equation}\label{eq:costclaim}
 d_*(\lambda,\mu)+\tau(\lambda)\geq\frac32
 \qquad
 \bigl(\lambda\in\mathcal S,\quad
 \mu\notin\mathcal S\cup\{\varnothing\}\bigr).
\end{equation}
Indeed, starting from $(1)$, the boundary costs to $(2,1)$,
$(1,1,1,1)$, $(3)$, and $(2,1,1)$ are respectively
$3/2$, $3/2$, $2$, and $2$.  Starting from $(1,1)$, the smallest boundary
cost is $1$, attained at $(2,1)$ and at $(1,1,1,1)$; adding
$\tau(1,1)=1/2$ gives $3/2$.  From $(1,1,1)$ the smallest boundary cost is
$1/2$, attained at $(1,1,1,1)$, and
$\tau(1,1,1)=1$.  From $(2)$ the smallest boundary cost is $1/2$, attained
at $(2,1)$, and $\tau(2)=1$.  Every path to an orbit beyond the boundary has
an additional edge of positive cost.  This proves \eqref{eq:costclaim}.

For a non-zero omitted orbit $\mu$, Lemma~\ref{lem:pathbound} and
\eqref{eq:costclaim} give
\begin{align*}
 \left|\ip{(K_{\ell,N}-\mu_{\ell,N})w_{\ell,N}}{\psi_\mu}\right|
 &=\left|\sum_{\lambda\in\mathcal S}a_{\lambda,N}
 \ip{K_{\ell,N}\psi_\lambda}{\psi_\mu}\right|\\
 &\leq C_\ell\sum_{\lambda\in\mathcal S}
 N^{-\tau(\lambda)}N^{\ell-d_*(\lambda,\mu)}
 \leq C_\ell N^{\ell-3/2}.
\end{align*}
There is no term involving $\mu_{\ell,N}w_{\ell,N}$ in this coefficient,
because $w_{\ell,N}$ has no component on an omitted orbit.  On the selected
four-dimensional space the residual vanishes, since $w_{\ell,N}$ is an exact
eigenvector of the compression.

Finally, $L_N^\ell$ has propagation at most $\ell$.  Starting from the four
fixed orbit types, it can therefore reach only finitely many partition types,
with their number depending on $\ell$ but not on $N$.  Summing the squares of
the preceding coefficient bounds over this finite set gives
\[
 \norm{(K_{\ell,N}-\mu_{\ell,N})w_{\ell,N}}^2
 =O_\ell(N^{2\ell-3}),
\]
which is \eqref{eq:residual}.
\end{proof}

We now combine sector selection, confinement, and the residual estimate to
identify the bottom of the full operator.

\begin{theorem}
\label{thm:general}
For every fixed $\ell\geq1$, as $N\to\infty$,
\begin{equation}\label{eq:full-expansion}
 C_\ell(N)=2^\ell N^\ell+\gamma_\ell N^{\ell-1}
 +\delta_\ell N^{\ell-2}+O_\ell(N^{\ell-3}),
\end{equation}
where $\gamma_\ell$ is given by \eqref{eq:gamma} and $\delta_\ell$ by
\eqref{eq:delta}.  The remainder is bounded in absolute value by
$A_\ell N^{\ell-3}$ for $N\geq N_0(\ell)$.
\end{theorem}

\begin{proof}
Let $\cH_N^{\mathrm{triv}}$ be the hyperoctahedrally invariant subspace.  The
closed-form symmetry argument in Section~2 shows that every isotypic component
reduces the Friedrichs operator $K_{\ell,N}$.

\smallskip\noindent\emph{Step (a): representations absent from the first
shell.}
If an irreducible representation has no copy in $P_1\cH_{0,N}$, its entire
isotypic component is contained in $Q_1\cH_{0,N}$.  Proposition
\ref{prop:confinement} therefore places its spectrum above
\[
 b_{\ell,N}=2^\ell\Lambda_N(1-O_\ell(N^{-1/2})),
\]
which is separated from $\Lambda_N$ by order $N^\ell$.  Such a representation
cannot contribute to the first-shell window.

\smallskip\noindent\emph{Step (b): one spectral root in each representation
present on the first shell.}
The coarse bound \eqref{eq:coarselower} applies to each restriction meeting
$P_1\cH_{0,N}$.  The unit vectors
\[
 p_{\mathrm{triv}}=\psi_{(1)},\qquad
 p_{\mathrm{ev}}=\frac12
 (\delta_{e_1}+\delta_{-e_1}-\delta_{e_2}-\delta_{-e_2}),\qquad
 p_{\mathrm{odd}}=\frac1{\sqrt2}(\delta_{e_1}-\delta_{-e_1})
\]
belong respectively to the trivial, even mean-zero, and odd first-shell
components.  Equation \eqref{eq:blocksPKP} gives a Rayleigh quotient
$\Lambda_N+O_\ell(N^{\ell-1})$ for each of them.  Thus the infimum of
the spectrum in each of the three isotypic sectors meeting
$P_1\cH_{0,N}$ lies in the window
$|z-\Lambda_N|\leq A_\ell N^{\ell-1}$.  In that window the $Q_1$ compression
is boundedly invertible, so Lemma~\ref{lem:schur} shows that the spectrum is
discrete and is obtained, with multiplicity, from the effective spectral
equation.  Proposition~\ref{prop:splitting} gives exactly one root in each of
the trivial, even mean-zero, and odd representations, with respective
asymptotics \eqref{eq:roottriv}--\eqref{eq:rootnontriv}.  Step~(a) excludes
all remaining representations from this window.

\smallskip\noindent\emph{Step (c): sector selection and rank one.}
Equation \eqref{eq:sector-gap} is strictly positive, so the global spectral
bottom belongs to $\cH_N^{\mathrm{triv}}$ for all sufficiently large $N$:
the difference of the two leading corrections is a positive multiple of
$N^{\ell-1}$, whereas both remainders are $o(N^{\ell-1})$.

Inside $\cH_N^{\mathrm{triv}}$, the first-shell subspace is one-dimensional.
Its orthogonal complement consists of invariant functions vanishing on $S_1$;
by Proposition~\ref{prop:confinement}, its quadratic form is bounded below by
the number $b_{\ell,N}$ defined in Step~(a).
Let $e_N=\psi_{(1)}$ in this sector.  If the spectral projection of the
restricted operator on $(-\infty,b_{\ell,N})$ had dimension at least two, its
range would contain a non-zero vector $f$ orthogonal to $e_N$.  In the
invariant sector, orthogonality to $e_N$ is equivalent to vanishing on $S_1$,
because invariant first-shell functions form the one-dimensional space
$\operatorname{span}\{e_N\}$.  The spectral theorem would give
$\overline{\mathfrak k}_{\ell,N}[f]<b_{\ell,N}\norm f^2$, contradicting the preceding form
bound.  Hence this spectral projection has dimension at most one.  For all
sufficiently large $N$,
$\mu_{\ell,N}=\Lambda_N+O_\ell(N^{\ell-1})<b_{\ell,N}$; the trial vector
$w_{\ell,N}$ has Rayleigh quotient $\mu_{\ell,N}$ and shows that the
projection is non-zero.  It therefore consists of
one isolated simple eigenvalue, which is the bottom of the invariant
restriction and, by the sector comparison above, the bottom of the full
operator.

We use Temple's inequality in the following standard form
\cite{Temple}: if $H$ is self-adjoint, $w$ is a unit vector,
$\mu=\ip{Hw}{w}$, and the spectrum of $H$ above its lowest eigenvalue is
contained in $[b,\infty)$ with $\mu<b$, then
\[
 \lambda_0(H)\geq
 \mu-\frac{\norm{(H-\mu)w}^2}{b-\mu}.
\]
The preceding one-dimensional spectral-projection argument verifies precisely
this hypothesis with $H=K_{\ell,N}|_{\cH_N^{\mathrm{triv}}}$ and
$b=b_{\ell,N}$.  Moreover,
$\ip{K_{\ell,N}w_{\ell,N}}{w_{\ell,N}}=\mu_{\ell,N}$ because
$w_{\ell,N}$ is an eigenvector of the compression and is supported on its
four orbit types.  Applying the inequality to $w_{\ell,N}$, Proposition
\ref{prop:delta} and Lemma~\ref{lem:residual} give
\[
 \mu_{\ell,N}
 -\frac{O_\ell(N^{2\ell-3})}{b_{\ell,N}-\mu_{\ell,N}}
 \leq C_\ell(N)\leq\mu_{\ell,N}.
\]
Here $b_{\ell,N}-\mu_{\ell,N}\geq c_\ell N^\ell$, and hence
\[
 C_\ell(N)=\mu_{\ell,N}+O_\ell(N^{\ell-3}),
\]
because the squared residual is $O_\ell(N^{2\ell-3})$.  Equation
\eqref{eq:full-expansion} now follows from \eqref{eq:ritzexp}.  All constants
used in the estimates depend only on the fixed order $\ell$; enlarging them
once gives the stated constants $A_\ell$ and $N_0(\ell)$.
\end{proof}

Substitution of the first three orders gives concrete special cases.

\begin{corollary}
\begin{align*}
 C_1(N)&=2N-4-\frac{20}{3N}+O(N^{-2}),\\
 C_2(N)&=4N^2-\frac{46}{3}N-\frac{7729}{540}+O(N^{-1}),\\
 C_3(N)&=8N^3-\frac{324}{7}N^2-\frac{30020}{4459}N+O(1).
\end{align*}
\end{corollary}

\begin{proof}
For $\ell=1$, formulas \eqref{eq:gamma} and \eqref{eq:delta} give
$\gamma_1=-4$ and $\delta_1=-20/3$; the remainder in
\eqref{eq:full-expansion} is $O(N^{-2})$.  For $\ell=2$, they give
$\gamma_2=-46/3$ and $\delta_2=-7729/540$, with remainder $O(N^{-1})$.
For $\ell=3$, they give $\gamma_3=-324/7$ and
$\delta_3=-30020/4459$, with remainder $O(1)$.  Inserting these values into
\eqref{eq:full-expansion} proves the three formulas.
\end{proof}

\section{Scope and further questions}

All asymptotic statements above are for fixed $\ell$ as $N\to\infty$; no
uniformity is asserted when $\ell$ grows with $N$.  The argument is analytic and
representation-theoretic, combining dimension-uniform estimates with exact
finite-dimensional orbit identities.

The formula depends on taking the full-lattice power $L_N^\ell$ before imposing
$u(0)=0$.  For a punctured Dirichlet Laplacian, the rank-one matrix
$\mathcal J_N$ is absent and the coefficient of $N^{\ell-1}$ changes.  It would
be interesting to carry out the analogous expansion in that convention and on
other highly symmetric graphs.  A second open problem is to obtain estimates
uniform in a growing order $\ell$.

\appendix

\section{Path bounds and the four-orbit calculation}

\subsection{Full-shell block estimates}\label{app:paths}

The orbit estimate in Lemma~\ref{lem:pathbound} has the following operator-norm
counterpart.  It is used in the two Feshbach reductions.

\begin{lemma}\label{lem:blockpath}
Fix finitely many orbit types and an integer $J$.  Decompose the corresponding
coordinate shells into hyperoctahedral orbits and let $E_\lambda$ denote the
full coordinate subspace supported on the orbit $\cO_\lambda$.  For
$0\leq j\leq J$,
\[
 \norm{E_\mu A_N^jE_\lambda}
 \leq C_JN^{a_j(\lambda,\mu)/2},
\]
with the same dimensional-edge count as in Lemma~\ref{lem:pathbound}.  The
constant is uniform over all irreducible components of the shell spaces.
\end{lemma}

\begin{proof}
For a single orbit edge, the bipartite incidence matrix has at most
$d_{\lambda\mu}$ non-zero entries in each row and at most
$d_{\mu\lambda}$ in each column.  The Schur test gives
\[
 \norm{E_\mu A_NE_\lambda}
 \leq\sqrt{d_{\lambda\mu}d_{\mu\lambda}}.
\]
Indeed, the Schur test bounds the norm by the square root of the product of
the maximal absolute row sum and maximal absolute column sum; all non-zero
entries of this incidence matrix equal one.  The counting in the proof of
Lemma~\ref{lem:adjacency} shows that the product of the two degrees is
$O_J(N)$ for a dimensional edge and $O_J(1)$ for an internal edge.  Hence the
corresponding block norms are $O_J(N^{1/2})$ and $O_J(1)$, respectively.

Insert the orbit projections between successive factors of $A_N^j$:
\[
 E_\mu A_N^jE_\lambda
 =\sum_{\lambda=\nu_0,\ldots,\nu_j=\mu}
 E_{\nu_j}A_NE_{\nu_{j-1}}\cdots
 E_{\nu_1}A_NE_{\nu_0},
\]
where the sum is over length-$j$ orbit paths.  Submultiplicativity of the
operator norm bounds a path with $a$ dimensional edges by $C_JN^{a/2}$.
Only finitely many orbit types and paths can be reached for fixed $J$, so
summing and maximizing $a$ proves the assertion.  If an invariant subspace is
inserted on both sides, the resulting operator is a compression of the full
block and therefore has no larger norm.  This proves the uniformity over all
irreducible components.
\end{proof}

Lemma~\ref{lem:shellblocks} applies this estimate after summing every orbit
reachable from $S_1$ or $(1,1)$ in at most $\ell$ steps.  Since that reachable
set is finite for fixed $\ell$, the constants remain uniform in $N$ and in the
irreducible component.  This is the operator-norm input used in both Schur
reductions.

\subsection{Exact moments and matrix expansion}\label{app:matrix}

Index the selected orbits by
\[
 \lambda_0=(1),\qquad \lambda_1=(1,1),\qquad
 \lambda_2=(1,1,1),\qquad \lambda_3=(2).
\]
Successive use of Lemma~\ref{lem:adjacency} gives
\begin{align}
 \ip{A_N\psi_{\lambda_0}}{\psi_{\lambda_1}}
 &=2\sqrt{N-1},
 &\ip{A_N\psi_{\lambda_1}}{\psi_{\lambda_2}}
 &=\sqrt{6(N-2)},\label{eq:exactedges}\\
 \ip{A_N\psi_{\lambda_0}}{\psi_{\lambda_3}}&=1,
 &\ip{A_N^2\psi_{\lambda_0}}{\psi_{\lambda_2}}
 &=2\sqrt{6(N-1)(N-2)}.\nonumber
\end{align}
The moments needed through weighted path cost $2$ are
\begin{align}
 \ip{A_N^2\psi_{\lambda_0}}{\psi_{\lambda_0}}
 &=6N-3,\label{eq:moments1}\\
 \ip{A_N^4\psi_{\lambda_0}}{\psi_{\lambda_0}}
 &=60N^2-90N+40,\label{eq:moments2}\\
 \ip{A_N^3\psi_{\lambda_0}}{\psi_{\lambda_1}}
 &=24(N-1)^{3/2},\label{eq:moments3}\\
 \ip{A_N^2\psi_{\lambda_1}}{\psi_{\lambda_1}}
 &=10N-14.\label{eq:moments4}
\end{align}
We give the exact path calculation.  The one-step neighbours of $(1)$ are the
origin, $(1,1)$, and $(2)$, with coefficients
$\sqrt{2N}$, $2\sqrt{N-1}$, and $1$.  The sum of their squares is
$2N+4(N-1)+1=6N-3$, proving \eqref{eq:moments1}.  A second adjacency step
gives the exact decomposition
\begin{align*}
 A_N^2\psi_{(1)}={}&(6N-3)\psi_{(1)}
 +2\sqrt{6(N-1)(N-2)}\,\psi_{(1,1,1)}\\
 &+3\sqrt{2(N-1)}\,\psi_{(2,1)}+\psi_{(3)}.
\end{align*}
The four orbit vectors are orthonormal.  Taking the squared norm of this
identity gives
\begin{align*}
 \ip{A_N^4\psi_{(1)}}{\psi_{(1)}}
 &=(6N-3)^2+24(N-1)(N-2)+18(N-1)+1\\
 &=60N^2-90N+40,
\end{align*}
which proves \eqref{eq:moments2}.

To compute \eqref{eq:moments3}, take the inner product of the preceding
decomposition with $A_N\psi_{(1,1)}$.  The relevant coefficients from
$(1,1)$ to $(1)$, $(1,1,1)$, and $(2,1)$ are
$2\sqrt{N-1}$, $\sqrt{6(N-2)}$, and $\sqrt2$.  Consequently,
\begin{align*}
 \ip{A_N^3\psi_{(1)}}{\psi_{(1,1)}}
 ={}&(6N-3)2\sqrt{N-1}
 +12(N-2)\sqrt{N-1}\\
 &+6\sqrt{N-1}
 =24(N-1)^{3/2}.
\end{align*}
Finally, the same three neighbours exhaust the adjacency edges from
$(1,1)$.  The sum of the squared coefficients is
\[
 4(N-1)+6(N-2)+2=10N-14,
\]
which proves \eqref{eq:moments4}.

Insert \eqref{eq:exactedges}--\eqref{eq:moments4} into
\[
 R^\ell L_N^\ell R^\ell
 =\sum_{j=0}^\ell(-1)^j\binom\ell j
 (2N)^{\ell-j}R^\ell A_N^jR^\ell.
\]
With $t=N^{-1/2}$ and after division by $(2N)^\ell$, one obtains
\begin{align*}
 \widehat M_{00}
 &=1+\frac32c_2t^2+
 \left(-\frac34c_2+\frac{15}{4}c_4\right)t^4+O_\ell(t^6),\\
 \widehat M_{01}
 &=-\sqrt {p_2}\,\ell t+
 \sqrt {p_2}\left(\frac\ell2-3c_3\right)t^3+O_\ell(t^5),\\
 \widehat M_{02}&=\frac{\sqrt{6p_3}}2c_2t^2+O_\ell(t^4),
 &\widehat M_{03}&=-\frac{p_2\ell}{2}t^2+O_\ell(t^4),\\
 \widehat M_{11}&=p_2+\frac52p_2c_2t^2+O_\ell(t^4),
 &\widehat M_{12}&=-\frac\ell2\sqrt{6p_2p_3}\,t+O_\ell(t^3),\\
 \widehat M_{22}&=p_3+O_\ell(t^2),
 &\widehat M_{33}&=p_2^2+O_\ell(t^2).
\end{align*}
We now check all entries not displayed to leading order.  The weighted
distances give
\[
 \widehat M_{13}=O_\ell(t^3),\qquad
 \widehat M_{23}=O_\ell(t^4).
\]
The first omitted corrections in $\widehat M_{02}$,
$\widehat M_{03}$, and $\widehat M_{11}$ have order $t^4$; those in
$\widehat M_{12}$, $\widehat M_{22}$, and $\widehat M_{33}$ have orders
$t^3$, $t^2$, and $t^2$, respectively.  Since the components with indices
$1,2,3$ begin at orders $t,t^2,t^2$, these omitted terms first enter the
first-coordinate equation at order $t^6$.  Equivalently, their first closed
perturbation paths have total order at least six.  Thus every coefficient
needed through order $t^4$ appears in the displayed matrix.

It remains to exclude an eigenvalue term of order $t^5$.  Every exact matrix
entry is a finite sum over orbit paths.  After the substitution $N=t^{-2}$,
each dimensional edge factor takes the form
$\sqrt{N-k}=t^{-1}\sqrt{1-kt^2}$, while internal edge factors are independent
of $N$.  Thus every entry of $\widehat M(t)$ is a finite sum of powers of $t$
multiplied by functions analytic in $t^2$ near zero.  In particular,
$\widehat M(t)$ is an analytic matrix family near $t=0$, so the simple
eigenvalue at $t=0$ has an analytic eigenvalue branch.

We now use a gauge symmetry of this analytic family.  A length-$j$ path with
$a$ dimensional edges contributes, after division by $(2N)^\ell$, a term
\[
 t^{2j-a}F(t^2),
\]
where $F$ is analytic near zero.  If
$k(\lambda)$ denotes the number of parts of a partition, then
\[
 a\equiv k(\lambda)-k(\mu)\pmod2
\]
for every path from $\lambda$ to $\mu$.  Define the diagonal matrix
$\mathcal D$ on the four selected orbits by
$\mathcal D\psi_\lambda=(-1)^{k(\lambda)}\psi_\lambda$.  The preceding parity
identity gives
\[
 \widehat M(-t)=\mathcal D\widehat M(t)\mathcal D.
\]
The eigenvalue issuing from $1$ is simple, so analytic perturbation theory and
this unitary equivalence imply $\lambda(-t)=\lambda(t)$.  Its expansion
contains only even powers.  Consequently the remainder after the $t^4$
coefficient is $O_\ell(t^6)$.
Substitution of
\[
 \lambda=1+gt^2+ht^4+O_\ell(t^6),\qquad
 v=(1,b_1t+b_3t^3,e_2t^2,d_2t^2)^T+O_\ell(t^4)
\]
gives successively
\begin{align*}
 &(p_2-1)b_1=A,\qquad
 (p_3-1)e_2=Gb_1-C,\qquad
 (p_2^2-1)d_2=-E,\\
 &g=a_0-Ab_1,\qquad
 (p_2-1)b_3=-B-(\kappa-g)b_1+Ge_2,\\
 &h=a_4-Ab_3+Bb_1+Ce_2+Ed_2.
\end{align*}
Here $p_2,p_3,c_j,a_0,a_4,A,B,C,E,\kappa,G$ are the quantities introduced in
the proof of Proposition~\ref{prop:delta}.  The equations agree with
\eqref{eq:b1}--\eqref{eq:h}; eliminating their auxiliary variables gives
\eqref{eq:hclosed}, and hence the closed formula \eqref{eq:delta}.  The convention
$\binom\ell j=0$ for $j>\ell$ makes the
calculation valid without modification for $\ell=1,2,3$.

Direct substitution for $\ell=1$ gives
$g=-2$, $h=-10/3$, hence
$\gamma_1=2g=-4$ and $\delta_1=2h=-20/3$.  For $\ell=2$ it gives
$\gamma_2=-46/3$ and $\delta_2=-7729/540$; for $\ell=3$ it gives
$\gamma_3=-324/7$ and $\delta_3=-30020/4459$.

\section*{Acknowledgements}

This work was supported by ANID--FONDECYT, grant No.~1260057.

\end{document}